%% file: main.tex
\documentclass[11pt]{amsart}
\usepackage{amsmath,amssymb,amsthm}
\usepackage{tikz}
\usepackage[letterpaper, margin=1in]{geometry}
\usepackage{todonotes}
\usetikzlibrary{decorations.markings,intersections,calc}
\usepackage{verbatim}
\usepackage{enumitem}
\usepackage{graphicx}
\usepackage{bm}
\usepackage{subcaption}

\usepackage[backref]{hyperref}
\hypersetup{
    colorlinks,
    linkcolor={red!60!black},
    citecolor={green!60!black},
    urlcolor={blue!60!black},
    pdftitle={..},
    pdfauthor={O-joung Kwon and Youngho Yoo}}

\usepackage{thmtools, thm-restate}

\newcommand\abs[1]{\ensuremath{\lvert #1\rvert}}

\newtheorem{theorem}{Theorem}[section]
\newtheorem{lemma}[theorem]{Lemma}

\newtheorem{proposition}[theorem]{Proposition}
\newtheorem{conjecture}[theorem]{Conjecture}

\newtheorem{claim}{Claim}[theorem]
\newtheorem*{claim*}{Claim}

\newenvironment{subproof}[1][Proof]{\begin{proof}[#1]}{\end{proof}}

\usetikzlibrary{decorations.pathmorphing}
\usetikzlibrary{decorations.markings}

\theoremstyle{definition}

\newcommand{\gen}[1]{\ensuremath{\langle #1\rangle}}
\newcommand{\branch}{\ensuremath{V_{\neq 2}}}

\def\lowfwd #1#2#3{{\mathop{\kern0pt #1}\limits^{\kern#2pt\raise.#3ex
\vbox to 0pt{\hbox{$\scriptscriptstyle\rightarrow$}\vss}}}}
\def\lowbkwd #1#2#3{{\mathop{\kern0pt #1}\limits^{\kern#2pt\raise.#3ex
\vbox to 0pt{\hbox{$\scriptscriptstyle\leftarrow$}\vss}}}}

\newcommand{\G}{\Gamma}
\newcommand{\g}{\gamma}
\newcommand{\mbn}{\mathbb{N}}

\newcommand{\mcf}{\mathcal{F}}
\newcommand{\mch}{\mathcal{H}}

\newcommand{\mcp}{\mathcal{P}}
\newcommand{\mcq}{\mathcal{Q}}
\newcommand{\mcr}{\mathcal{R}}
\newcommand{\mcs}{\mathcal{S}}
\newcommand{\mct}{\mathcal{T}}

\newcommand{\mbz}{\mathbb{Z}}

\newcommand{\ep}{Erd\H{o}s-P\'osa }
\newcommand{\hi}{half-integral }

\begin{document}

\title{Erd\H{o}s-P\'osa property of $A$-paths in unoriented group-labelled graphs}

\author{O-joung Kwon}
\address[Kwon]{\small Department of Mathematics, Hanyang~University, Seoul,~South~Korea and Discrete Mathematics Group, Institute~for~Basic~Science~(IBS), Daejeon,~South~Korea.}
\email{{ojoungkwon@hanyang.ac.kr}}

\author{Youngho Yoo}
\address[Yoo]{\small Department of Mathematics and Statistics, University of Alaska Fairbanks, AK, USA}
\email{{yyoo2@alaska.edu}}
\thanks{O. Kwon is supported by the National Research Foundation of Korea (NRF) grant funded by the Ministry of Education (No. NRF-2021K2A9A2A11101617 and No. RS-2023-00211670) and supported by the Institute for Basic Science (IBS-R029-C1).}
\begin{abstract}
    We characterize the obstructions to the Erd\H{o}s-P\'osa property of $A$-paths in unoriented group-labelled graphs. As a result, we prove that for every finite abelian group $\Gamma$ and for every subset $\Lambda$ of $\Gamma$, the family of $\Gamma$-labelled $A$-paths whose lengths are in $\Lambda$ satisfies the half-integral Erd\H{o}s-P\'osa property. Moreover, we give a characterization of such $\Gamma$ and $\Lambda\subseteq\Gamma$ for which the same family of $A$-paths satisfies the full Erd\H{o}s-P\'osa property.
\end{abstract}
\maketitle

\section{Introduction}
A family $\mcf$ of graphs is said to satisfy the \emph{\ep property} if there exists a function $f$, called an \emph{\ep function for $\mcf$}, such that for every positive integer $k$, every graph $G$ contains either $k$ (vertex-)disjoint subgraphs in $\mcf$ or a set of at most $f(k)$ vertices intersecting every subgraph of $G$ in $\mcf$. This property is named after Erd\H{o}s and P\'osa who proved  \cite{ErdosP1965} in 1965 that the family of cycles satisfies this property, and has since been studied for numerous other families of graphs. 

For instance, Lov\'asz and Schrijver observed (see \cite{Thomassen1988}) that the family of odd cycles does \emph{not} satisfy the \ep property due to certain projective planar grids; on the other hand, even cycles \cite{DejterN1988} and, more generally, cycles of length 0 modulo $m$ for any fixed positive integer $m$ \cite{Thomassen1988} do satisfy the \ep property.
The question of when the \ep property holds for cycles of length $\ell$ modulo $m$, for fixed integers $\ell$ and $m>0$, was posed by Dejter and Neumann-Lara \cite{DejterN1988} in 1988 and recently resolved by Gollin, Hendrey, Kwon, Oum, and Yoo \cite{gollin2022unified} with a characterization of the \ep property of ``allowable'' cycles in a more general setting of unoriented group-labelled graphs.

Robertson and Seymour \cite{robertson1986graph} generalized the \ep theorem in different direction and showed in 1986 that, for a fixed graph $H$, the family of graphs that contain $H$ as a minor satisfies the \ep property if and only if $H$ is planar.
Liu, Postle, and Wollan (see \cite{liu2022packing}) proved an analogous characterization of the \ep property of topological minors.

Another widely studied class of graph families includes various families of $A$-paths: for a vertex set $A$, an \emph{$A$-path} is a path with at least one edge such that it has both endpoints in $A$ and is internally disjoint from $A$. 
Gallai \cite{gallai1961maximum} showed in 1961 that $A$-paths satisfy the \ep property, and this result was generalized in several ways \cite{bruhn2018frames, bruhn2022packing, Chudnovsky2006, geelen2009odd, geelen2021disjoint, Mader1978, pap2007packing, thomas2023packingApaths, wollan2010packing, yamaguchi2016packing}.
For example, Thomas and Yoo \cite{thomas2023packingApaths} proved an analogue of Dejter and Neumann-Lara's question for $A$-paths by characterizing when the \ep property holds for $A$-paths of length $\ell$ modulo $m$.

In many cases, when a graph family $\mcf$ fails to satisfy the \ep property, one can salvage a weaker \emph{half-integral \ep property} where, rather than insisting on a \emph{packing} (a set of $k$ disjoint subgraphs of $G$ in $\mcf$), we settle for a \emph{half-integral packing} (a multiset of members of $\mcf$ in $G$ such that no vertex belongs to more than two elements of the multiset).

For instance, while odd cycles do not satisfy the \ep property, Reed \cite{Reed1999} proved a structure theorem showing that the only obstructions to the \ep property of odd cycles are certain projective planar grids called \emph{Escher walls}, and derived as a consequence that odd cycles satisfy the half-integral \ep property.
Far-reaching generalizations of Reed's structure theorem and corresponding half-integral \ep results for cycles in group-labelled graphs were proved in \cite{gollin2024unified, gollin2022unified}.
Similarly, while there are many graphs $H$ for which the family of graphs containing $H$ as a (topological) minor does not satisfy the \ep property, Liu \cite{liu2022packing} proved that the \hi \ep property holds for all graphs $H$ for an even stronger notion of \emph{rooted} topological minors.

In contrast, there were no half-integral \ep results for families of $A$-paths that fail the \ep property, until very recently; Chekan et al.~\cite{chekan2024half} showed that, for vertex sets $S$ and $T$, the \hi \ep property holds for the family of paths starting in $S$ and ending in $T$ of odd length (in fact, they prove this for paths of \emph{nonzero} length in \emph{oriented} group-labelled graphs, labelled by elements from a finite group).

In this paper, we characterize the topological obstructions to the \ep property of $A$-paths in \emph{unoriented} group-labelled graphs, for finite abelian groups.
Our structure theorem gives a general \hi \ep result for $A$-paths, as well as a characterization of the full \ep property.

Let $\G$ be a finite abelian group with additive operation and identity element 0. An (\emph{unoriented}) \emph{$\G$-labelled graph}\footnote{Henceforth, all group-labelled graphs are assumed to be unoriented.} is a pair $(G,\g)$ of a graph $G$ and a \emph{$\G$-labelling} $\g:E(G)\to\G$.
The \emph{$\g$-length} of a subgraph $H$ of $G$ is defined to be $\sum_{e\in E(H)}\g(e)$ and is denoted $\g(H)$. 

Let $\Lambda$ be a subset of $\G$. We say that a subgraph $H$ of $G$ is \emph{$\Lambda$-allowable} if $\g(H)\in\Lambda$.
We say that the set $\Lambda$ satisfies the \emph{Erd\H{o}s-P\'osa condition (in $\G$)}  if both of the following conditions are satisfied:
    \begin{enumerate}
        [label=(EP\arabic*)]
        \item \label{item:ep1}
        For all $a,b,c\in\G$, if $a+b+c\in\Lambda$, then either $a+b\in\Lambda$, $(2a+\langle c\rangle)\cap\Lambda\neq\emptyset$, or $(2b+\langle c\rangle)\cap\Lambda\neq\emptyset$.
        \item \label{item:ep2} 
        For all $a,b,c\in\G$, if $2a+b+c\in \Lambda$, then either $(2a+\langle b\rangle)\cap\Lambda\neq\emptyset$ or $(2a+\langle c\rangle)\cap\Lambda \neq\emptyset$.
    \end{enumerate}
It is easy to see from Figure \ref{fig:obstructions} that if \ref{item:ep1} or \ref{item:ep2} is violated, then there is a $\Gamma$-labelled graph that does not contain two disjoint $\Lambda$-allowable $A$-paths but requires many vertices to intersect every $\Lambda$-allowable $A$-path. Hence, the \ep condition is necessary for the \ep property of $\Lambda$-allowable $A$-paths (see also Proposition \ref{prop:noEPC}). 
On the other hand, note that both graphs in Figure \ref{fig:obstructions} contain a large half-integral packing of $\Lambda$-allowable $A$-paths.

Our main structure theorem (Theorem \ref{thm:mainApathstructure}) shows that the only obstructions to the \ep property of $\Lambda$-allowable $A$-paths are essentially those in Figure \ref{fig:obstructions}. This implies that the \ep condition is also sufficient for the \ep property and that the \hi \ep property always holds for $\Lambda$-allowable $A$-paths in finite abelian groups.
These are the main results of our paper.
\begin{theorem}
    \label{thm:mainep}
    Let $\G$ be a finite abelian group and let $\Lambda\subseteq \G$. Let $\mcf$ denote the family of $\Lambda$-allowable $A$-paths. Then
    \begin{enumerate}
        \item $\mcf$ satisfies the \hi \ep property, and \label{thm:mainep1}
        \item $\mcf$ satisfies the \ep property if and only if $\Lambda$ satisfies the \ep condition. \label{thm:mainep2}
    \end{enumerate}
\end{theorem}
Theorem \ref{thm:mainep}\eqref{thm:mainep1} generalizes the recent odd path result of Chekan et al. \cite{chekan2024half} (see Section \ref{subsubsec:weakABpaths})\footnote{We remark that our result is not a strict strengthening of the full result of Chekan et al. \cite{chekan2024half} which is for nonzero paths in \emph{oriented} group-labelled graphs. While the oriented and the unoriented settings coincide for $\mathbb{Z}/2\mathbb{Z}$ (and more generally for groups whose nonzero elements all have order 2), the two settings are incomparable in general.} and partially answers a question of Gollin et al. \cite[Question 1]{gollin2024unified} who asked whether the \hi \ep property holds for the family of $\Lambda$-allowable $A$-paths when $\Gamma$ is the product of $m$ abelian groups $\Gamma_1\times\dots\times\Gamma_m$, where each $\Gamma_i$ may be infinite, such that $\Lambda = (\G_1-\Omega_1)\times \dots\times (\G_m-\Omega_m)$ for some $\Omega_i\subseteq\Gamma_i$ with $|\Omega_i| \leq \omega$, for fixed $m$ and $\omega$.
We suspect that our techniques for finite groups can be extended to this more general setting, but our proof relies on  some tools from \cite{thomas2023packingApaths} that depend on the finiteness of $\G$, and the bound we obtain on the \ep function grows with $|\G|$.
Our proof is constructive and our bound is computable, in contrast to the non-constructive proof of \cite{chekan2024half}.
Moreover, for finite groups, we obtain a stronger result than what was asked in \cite[Question 1]{gollin2024unified} in the sense that we characterize the full \ep property by the \ep condition in Theorem \ref{thm:mainep}\eqref{thm:mainep2}; this also generalizes the characterization of Thomas and Yoo \cite{thomas2023packingApaths} (see Section \ref{subsubsec:mod}). We conjecture that this characterization also holds in the more general setting with infinite groups; an analogous result for cycles was proved in \cite{gollin2022unified}.
\begin{conjecture} \label{conj}
    Let $m$ and $\omega$ be positive integers and, for each $i\in[m]$, let $\G_i$ be an abelian group and let $\Omega_i$ be a subset of $\G_i$ with $|\Omega_i|\leq \omega$. Let $\Lambda = (\G_1-\Omega_1)\times \dots\times (\G_m-\Omega_m)$, and let $\mcf$ denote the family of $\Lambda$-allowable $A$-paths. Then
    \begin{enumerate}
        \item $\mcf$ satisfies the \hi \ep property, and \label{conj1}
            \item $\mcf$ satisfies the \ep property if and only if $\Lambda$ satisfies the \ep condition. \label{conj2}
    \end{enumerate}
\end{conjecture}
Conjecture \ref{conj} is a strengthening of the question of Gollin et al. \cite[Question 1]{gollin2024unified}, which is equivalent to Conjecture \ref{conj}\eqref{conj1}.
The necessity of the \ep condition in Conjecture \ref{conj}\eqref{conj2} is again easy to see from Figure \ref{fig:obstructions}.

Note that neither \ref{item:ep1} nor \ref{item:ep2} can be removed from the definition of the \ep condition.

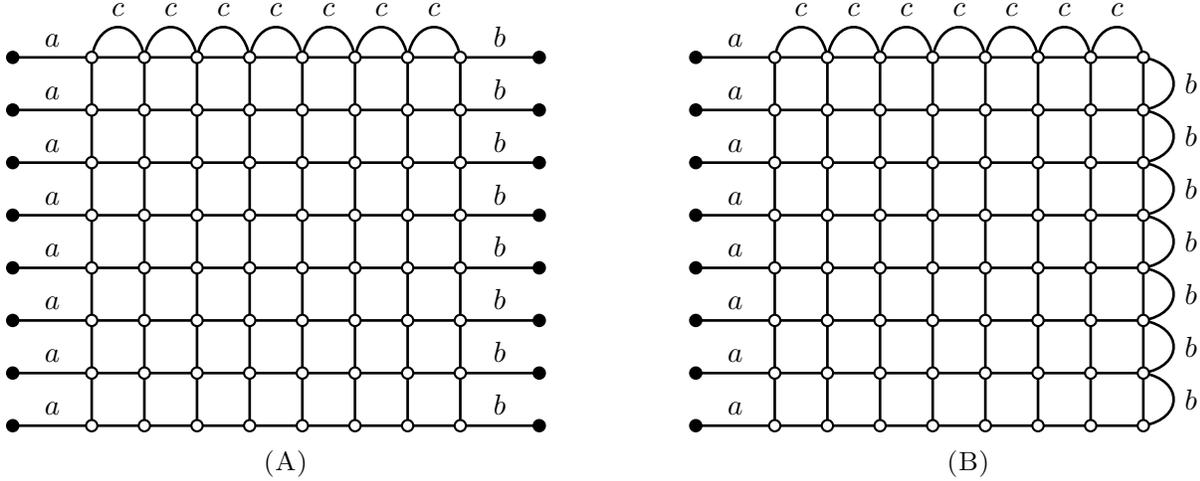
\begin{figure}
    \centering
    \begin{subfigure}{0.45\textwidth}
        \input{FigObs1.tex}
        \caption{} \label{fig:obs1}
    \end{subfigure}
      \hspace*{\fill}
    \begin{subfigure}{0.45\textwidth}
        \input{FigObs2.tex}
        \caption{}\label{fig:obs2}
    \end{subfigure}
    \caption{The black vertices constitute $A$ and all edges are labelled 0 unless otherwise indicated. If $\Lambda$ does not satisfy \ref{item:ep1}, then every $\Lambda$-allowable $A$-path in (\subref{fig:obs1}) has one endpoint on the left and one on the right and contains at least one of the top edges labelled $c$. If $\Lambda$ does not satisfy \ref{item:ep2}, then every $\Lambda$-allowable $A$-path in (\subref{fig:obs2}) has both endpoints on the left and contains at least one edge labelled $c$ and at least one edge labelled $b$. In either case, no two $\Lambda$-allowable $A$-paths are disjoint.}
    \label{fig:obstructions}
\end{figure}

\begin{proposition}
    There exists a finite abelian group $\Gamma$ and $\Lambda\subseteq\G$ satisfying \ref{item:ep1} but not \ref{item:ep2}, and there exists a finite abelian group $\G'$ and $\Lambda'\subseteq \G'$ satisfying \ref{item:ep2} but not \ref{item:ep1}.
\end{proposition}
\begin{proof}
    Let $\G=\mbz/15\mbz$ and let $\Lambda = \{1,2,4,7,8,11,13,14\}$;\footnote{For simplicity, we denote the elements of $\mbz/m\mbz$ by $0,1,2,\dots,m-1$ and treat them as integers as convenient.} that is, $\Lambda$ is the set of integers in $\G$ relatively prime to 15. Then \ref{item:ep2} is not satisfied; indeed, if $a=0$, $b=3$, and $c=5$, then $2a+b+c = 8\in\Lambda$, but $2a+\langle b\rangle = \langle 3\rangle$ and $2a+\langle c\rangle = \langle 5\rangle$, neither of which intersects $\Lambda$. 
    Now let us show that \ref{item:ep1} is satisfied. Let $a,b,c\in\G$ and suppose that $a+b+c\in\Lambda$ and $a+b\not\in\Lambda$. Then $c\neq 0$. If $c\in\Lambda$, then $\G=\langle c\rangle = 2a+\langle c\rangle = 2b+\langle c\rangle$, so we may assume that $\gcd(c,15)\in\{3,5\}$. Since the only coset of $\langle c\rangle$ disjoint from $\Lambda$ is $\langle c\rangle$, if $(2a+\langle c\rangle) \cap \Lambda = \emptyset$, then $2a\in\langle c\rangle$, which implies that $a\in\langle c\rangle$ since $|\G/\langle c\rangle| \in\{3,5\}$. So we may assume that $a\in\langle c\rangle$ and, similarly, that $b\in\langle c\rangle$. But this implies that $a+b+c\in\langle c\rangle$ which is disjoint from $\Lambda$, a contradiction.

    Now let $\G' = \mbz/6\mbz$ and $\Lambda' = \{4\}$. Then \ref{item:ep1} is not satisfied; indeed, if $a=0$, $b=1$, and $c=3$, then $a+b+c=4 \in \Lambda'$, but $a+b=1$, $2a+\langle c\rangle = \langle 3\rangle$, and $2b+\langle c\rangle = 2+\langle 3\rangle$; the latter two sets clearly do not contain 4.
    Now let us show that \ref{item:ep2} is satisfied. Let $a,b,c\in\G'$ and suppose that $2a+b+c=4$. We may assume that $b,c\neq 0$ since otherwise we have $2a+c=4$ or $2a+b=4$. We may also assume that $b,c\notin \{1,5\}$ since otherwise we have $\langle b\rangle=\G'$ or $\langle c\rangle=\G'$. Hence $b,c\in\{2,3,4\}$.
    Since $b+c = 4-2a \in\langle 2\rangle$, we have either $b=c=3$ or $b,c\in\{2,4\}$. In the first case, we have $4=2a+b+c=2a+2b\in 2a+\langle b\rangle$. In the second case, we have $4=2a+b+c \in 2a+\langle b\rangle$. 
\end{proof}

\subsection{Special cases}
Let us discuss some natural constraints on $A$-paths that can be encoded as instances of Theorem \ref{thm:mainep}.
\subsubsection{Modularity constraints} \label{subsubsec:mod}
Let $m$ be a positive integer and let $\Lambda\subseteq \mbz/m\mbz$. By Theorem \ref{thm:mainep}\eqref{thm:mainep1}, the family of $A$-paths whose length is in a congruence class in $\Lambda$ satisfies the \hi \ep property. In particular, for all integers $\ell$, the family of $A$-paths of length $\ell$ modulo $m$ satisfies the \hi \ep property. Thomas and Yoo \cite{thomas2023packingApaths} proved the following characterization of the \ep property of $\Lambda$-allowable $A$-paths when $|\Lambda|=1$.
\begin{theorem}
    [{\cite[Theorem 1.4]{thomas2023packingApaths}}] \label{thm:tymain}
    Let $\G$ be an abelian group and let $\ell\in \G$. Then the family of $\{\ell\}$-allowable $A$-paths satisfies the \ep property if and only if
    \begin{itemize}
        \item $\G \cong (\mathbb{Z}/2\mathbb{Z})^k$ where $k\in\mathbb{N}$ and $\ell=0$,
        \item $\G \cong \mathbb{Z}/4\mathbb{Z}$ and $\ell\in\{0,2\}$, or
        \item $\G \cong \mathbb{Z}/p\mathbb{Z}$ where $p$ is prime (and $\ell\in\G$ is arbitrary).
    \end{itemize}
\end{theorem}
    It is not difficult to see that that if none of the three conditions in Theorem \ref{thm:tymain} are satisfied, then the \ep property fails (in fact, \ref{item:ep1} fails; see \cite[\S2.3, 2.4]{thomas2023packingApaths}). The proof of the converse is much more involved; that the \ep property holds for $A$-paths of length 0 modulo 4 is the main result of Bruhn and Ulmer \cite{bruhn2022packing}, and that the \ep property holds for $A$-paths of length 0 modulo a prime is the main result of \cite{thomas2023packingApaths}. However, it is easy to see that if any one of the three conditions in Theorem \ref{thm:tymain} are satisfied, then the \ep condition holds for $\Lambda=\{\ell\}$. Hence, Theorem \ref{thm:mainep} generalizes Theorem \ref{thm:tymain}.

\subsubsection{$A$-paths through specified edges or vertices}
For a fixed positive integer $k$, let $F_1,\dots,F_k$ be edge sets, and suppose that we are interested in the family $\mcf$ of $A$-paths that contain at least one edge in $F_i$ for each $i\in[k]$. Then $\mcf$ can be modelled as the $\Lambda$-allowable $A$-paths in $(\mbz/2\mbz)^k$-labelled graphs as follows.

Define $\Gamma = (\mbz/2\mbz)^k$ and let $\{g_1,\dots,g_k\}$ be a set of generators of $\Gamma$.
Let $G$ be a simple graph with $A\subseteq V(G)$ and $F_1,\dots,F_k\subseteq E(G)$. Let $(G',\g)$ be the $\G$-labelled graph obtained from $V(G)$ as follows.  
For each edge $e=uv\in E(G)$, let $\mcs_e$ denote the power set of $\{i\in[k]:e\in F_i\}$.
For each $S\in\mcs_e$, we put an edge $e_S$ in $G'$ joining $u$ and $v$ with label $\g(e_S) = \sum_{i\in S} g_i$. 
Note that for each $e=uv\in E(G)$, there are $2^{|\mcs_e|}$ parallel edges joining $u$ and $v$ in $G'$.

Let $\Lambda\subseteq\Gamma$ be the set consisting of the single element $\sum_{i\in[k]}g_i$.
Clearly, every $\Lambda$-allowable $A$-path in $G'$ corresponds to an $A$-path in $G$ in $\mcf$ with the same sequence of vertices.
Conversely, for each $A$-path in $G$ in $\mcf$, there is a $\Lambda$-allowable $A$-path in $G'$ with the same sequence of vertices with the appropriate choices of edges.
Hence, by Theorem \ref{thm:mainep}\eqref{thm:mainep1}, $\mcf$ satisfies the \hi \ep property.

Similarly, if $U_1,\dots,U_k$ are vertex sets, then the family $\mcf$ of $A$-paths containing at least one vertex in $U_i$ for each $i\in[k]$ satisfies the \hi \ep property; this can be obtained by setting $F_i$ to be the set of edges incident to $U_i$ for each $i\in[k]$.

In both of the above settings, $\mcf$ satisfies the full \ep property if and only if $k=1$. The case $k=1$ is implied by Theorem \ref{thm:mainep}\eqref{thm:mainep2} and also by previously known results \cite{Chudnovsky2006, wollan2010packing}. If $k>1$, then the obstruction in Figure \ref{fig:obs2}, with $F_1$ and $F_2=\dots=F_k$ being the sets of edges labelled $b$ and $c$ respectively (or in the vertex case, letting $U_1$ and $U_2=\dots=U_k$ be the sets of vertices created by subdividing once every edge labelled $b$ and $c$ respectively), shows that $\mcf$ does not satisfiy the \ep property.

\subsubsection{$A$-$B$-paths}
Let $A$ and $B$ be vertex sets. An \emph{$A$-$B$-path} is a path with one endpoint in $A$, the other endpoint in $B$, and internally disjoint from $A\cup B$.
Let $\Gamma$ be a finite abelian group and let $\Lambda\subseteq \Gamma$. Suppose that we are interested in $\Lambda$-allowable $A$-$B$-paths in $\Gamma$-labelled graphs. 

Define $\G' = \G\times (\mbz/2\mbz)^2$.
Given a $\G$-labelled graph $(G,\g)$, we define a $\G'$-labelled graph $(G',\g')$ as follows. Set $V(G')=V(G)$. For each edge $e=uv\in E(G)$, we assign $\g'(e)$ as follows.
\begin{itemize}
    \item If neither $u$ nor $v$ is in $A\cup B$, then add $e$ to $E(G')$  with label $\g'(e) = (\g(e),(0,0))$.
    \item If exactly one of $u,v$ is in $A\cup B$, say $u$, then 
    \begin{itemize}
        \item if $u\in A$, add an edge $e_1$ joining $u$ and $v$ to $E(G')$ with label $\g'(e_1) = (\g(e),(1,0))$, and
        \item if $u\in B$, add an edge $e_2$ joining $u$ and $v$ to $E(G')$ with label $\g'(e_2) = (\g(e),(0,1))$.
    \end{itemize}
    Note that if $u\in A\cap B$, then $e_1$ and $e_2$ are parallel edges in $G'$.
    \item If $u\in A$ and $v\in B$ or $u\in B$ and $v\in A$, then add $e$ to $E(G')$ with label $(\g(e),(1,1))$.
    Note that if $u,v$ are both $A-B$ or both in $B-A$, then $e$ does not belong to any $A$-$B$-path in $G$.
\end{itemize}
Define $\Lambda' = \{(g,(1,1)):g\in\Lambda\}$.
Then $\Lambda$-allowable $A$-$B$-paths in $(G,\g)$ correspond to $\Lambda'$-allowable $(A\cup B)$-paths in $(G',\g')$ and vice versa.
Hence, by Theorem \ref{thm:mainep}\eqref{thm:mainep1}, the family of $\Lambda$-allowable $A$-$B$-paths satisfy the \hi \ep property.

We remark that the full \ep property does not hold for $\Lambda$-allowable $A$-$B$-paths whenever $\emptyset \neq \Lambda \subsetneq \G$. Indeed, if $\emptyset \neq \Lambda \subsetneq \G$, then there exist $a \notin \Lambda$ and $c\in \G$ such that $a+c \in \Lambda$.
Let $b=0$. Then the obstruction in Figure \ref{fig:obs1}, with $A$ being the set of vertices of degree 1 incident to edges labelled $a$ and $B$ being the set of vertices of degree 1 incident to edges labelled $b$, shows that $\Lambda$-allowable $A$-$B$-paths do not satisfy the \ep property.
If $\Lambda=\G$, then all $A$-$B$-paths are $\Lambda$-allowable $A$-$B$-paths, so we have the \ep property in this case by Menger's theorem.

\subsubsection{Weak $A$-$B$-paths} \label{subsubsec:weakABpaths}
A \emph{weak $A$-$B$-path} is a path with one endpoint in $A$ and the other endpoint in $B$; its internal vertices are allowed to be in $A\cup B$.
Chekan et al. \cite{chekan2024half} showed that weak $A$-$B$-paths of odd length satisfy the \hi \ep property. We generalize this to $\Lambda$-allowable weak $A$-$B$-paths for arbitrary finite abelian groups $\Gamma$ and $\Lambda\subseteq\Gamma$.

Let $(G,\g)$ be a $\G$-labelled graph and let $A,B\subseteq V(G)$. Let $G'$ be the graph obtained from $G$ as follows. For each vertex $a\in A$, add a new vertex $a'$ and an edge $aa'$. Define $A' = \{a':a\in A\}$. 
For each vertex $b\in B$, add a new vertex $b'$ and an edge $bb'$. Define $B' = \{b':b\in B\}$. 
Define $\G' = \G\times (\mbz/2\mbz)^2$ and a $\G'$-labelling $\g'$ of $G'$ as follows. For $e\in E(G')$,
\begin{align*}
    \g'(e)=
    \left\{
	\begin{array}{ll}
		(\g(e),(0,0))  & \mbox{if $e \in E(G)$,}\\
		  (0,(1,0)) & \mbox{if $e$ is incident to $A'$,} \\
            (0,(0,1)) & \mbox{if $e$ is incident to $B'$.}
	\end{array}
\right.
\end{align*}
Define $\Lambda' = \{(g,(1,1)):g\in \Lambda\}$.
Then $\Lambda$-allowable weak $A$-$B$-paths in $(G,\g)$ correspond to $\Lambda'$-allowable $(A'\cup B')$-paths in $(G',\g')$. Hence, by Theorem \ref{thm:mainep}\eqref{thm:mainep1}, the family of $\Lambda$-allowable weak $A$-$B$-paths satisfies the \hi \ep property.

Note that multiple constraints can be combined by taking direct products of the groups and the $\Lambda$'s. For example, if $A, B, U$ are vertex sets and $F$ is an edge set, then the family of $A$-$B$-paths that contain at least one vertex of $U$, contain at least one edge of $F$, and has length 0 modulo 3 satisfies the \hi \ep property.

\subsubsection{More general endpoint constraints}
Let $\mcs=\{S_1,\dots,S_k\}$ be a partition of $A$ with a fixed number $k$ of parts.
Let $H$ be a graph with vertex set $\mcs$ with possibly loops (but no parallel edges).
Let us say that an $A$-path $P$ with endpoints $a,b$ is \emph{$H$-feasible} if the part of $\mcs$ containing $a$ is adjacent in $H$ to the part of $\mcs$ containing $b$.
Note that $A$-$B$-paths are the special case where $k=2$ or $k=3$, depending on whether $A\cap B=\emptyset$.
Moreover, if $H\cong K_s$, then $H$-feasible paths correspond to ``$\mcs$-paths'' considered by Mader \cite{Mader1978} (albeit here we need $|\mcs|$ bounded). 
It is easy to see that $\Lambda$-allowable $H$-feasible paths satisfy the \hi \ep property by encoding them in a similar manner to $A$-$B$-paths. In fact, we can put different constraints on the edges of $H$.

Let $\Gamma$ be a finite abelian group and let $\Lambda$ be a function that maps each edge $f$ of $H$ to a subset $\Lambda(f)$ of $\Gamma$. Let us say that a $\G$-labelled nontrivial path $P$ with endpoints in $A$ is \emph{$(H,\Lambda)$-feasible} if the following condition is satisfied: 
\begin{itemize}
    \item Let $a,b$ denote the endpoints of $P$ and let $S_i,S_j$ denote the parts of $\mcs$ containing $a,b$ respectively. Then $f:=\{S_i,S_j\}\in E(H)$ and the length of $P$ is in $\Lambda(f)$.
\end{itemize}
Then the family of $(H,\Lambda)$-feasible $A$-paths satisfies the \hi \ep property, as we now show.

Let $\G' = \G \times (\mbz/3\mbz)^k$, and let $\{g_1,\dots,g_k\}$ be a set of generators of $(\mbz/3\mbz)^k$. 
Given a $\G$-labelled graph $(G,\g)$, we define a $\G'$-labelling $\g'$ of $G$ as follows. For each edge $e=uv\in E(G)$,
\begin{itemize}
    \item if neither $u$ nor $v$ is in $A$, then define $\g'(e) = (\g(e),0)$.
    \item if exactly one of $u$ and $v$ is in $A$, say $u \in S_i$, then define $\g'(e) = (\g(e),g_i)$.
    \item if both $u$ and $v$ are in $A$, say $u\in S_i$ and $v\in S_j$ (where possibly $i=j$), then define $\g'(e) = (\g(e),g_i+g_j)$.
\end{itemize}
Define $\Lambda' = \{(g,g_i+g_j): f:=\{S_i,S_j\}\in E(H) \text{ and } g\in \Lambda(f)\}$. Then it is easy to see that an $A$-path in $(G,\g)$ is $(H,\Lambda)$-feasible if and only if it is $\Lambda'$-allowable in $(G,\g')$.
Hence, by Theorem \ref{thm:mainep}\eqref{thm:mainep1}, the family of $(H,\Lambda)$-feasible $A$-paths satisfies the \hi \ep property.

As an example, if $A,B,C$ are disjoint vertex sets, then we can ask for $(A\cup B\cup C)$-paths $P$ such that
\begin{itemize}
    \item $P$ has its endpoints in distinct sets of $\{A,B,C\}$,
    \item if $P$ has one endpoint in $B$, then $P$ has length 1 modulo 3,
    \item if $P$ has one endpoint in $A$ and one endpoint in $C$, then $P$ has length 2 modulo 3.
\end{itemize}
These $(A\cup B\cup C)$-paths satisfy the \hi \ep property.

One can adapt the argument for weak $A$-$B$-paths to obtain the analogous result for $(H,\Lambda)$-feasible \emph{weak $A$-paths}, where we allow the paths to intersect $A$ internally.

We also remark that Theorem \ref{thm:mainep} (and hence all of the special cases above) holds for \emph{long} $\Lambda$-allowable $A$-paths as well; that is, for every fixed positive integer $\ell$, the family of $\Lambda$-allowable $A$-paths with at least $\ell$ edges satisfies the \hi \ep property, and satisfies the full \ep property if and only if $\Lambda$ satisfies the \ep condition. We briefly discuss how the proof of Theorem \ref{thm:mainep} can be modified to accomodate long paths at the end of Section \ref{sec:obstructions}.

\subsection{Paper outline}
We begin with basic definitions and preliminaries in Section \ref{sec:prelim}. In Section \ref{sec:obstructions}, we describe the obstructions to the \ep property of $A$-paths in group-labelled graphs, state our main structure theorem (Theorem \ref{thm:mainApathstructure}), and show that Theorem \ref{thm:mainApathstructure} implies Theorem \ref{thm:mainep}.
In Section \ref{sec:lemmas}, we give an outline of the proof of Theorem \ref{thm:mainApathstructure} and prove several lemmas.
In Section \ref{sec:proof}, we combine these lemmas to complete the proof of Theorem \ref{thm:mainApathstructure}.

\section{Preliminaries} \label{sec:prelim}

In this paper, all graphs are undirected graphs that have no loops, but may have multiple edges.
For a graph~$G$, we denote by~${V(G)}$ and~${E(G)}$ the vertex set and the edge set of~$G$, respectively; for a set~$A$ of vertices in~$G$, we denote by~${G - A}$  the graph obtained from~$G$ by deleting all the vertices in~$A$, 
and denote by~${G[A]}$ the subgraph of~$G$ induced by~$A$.
For two graphs~$G$ and~$H$, let 
${G \cup H := (V(G) \cup V(H), E(G) \cup E(H))}$ and $G \cap H := (V(G) \cap V(H), E(G) \cap E(H)).$

Let~${\mathbb N}$ denote the set of all non-negative integers, and let $\mathbb Z$ denote the set of all integers. For an integer~$m$, let~${[m]}$ denote the set of positive integers at most~$m$.

\subsection{3-blocks}
Let $G$ be a graph. A subset $U\subseteq V(G)$ is \emph{2-inseparable} if no set of at most two vertices of $G$ disconnects two vertices of $U$ in $G$.
A 2-inseparable set is \emph{maximal} if it is not properly contained in another 2-inseparable set.
Observe that if $U$ is a maximal 2-inseparable set, then for every component $C$ of $G-U$, there are at most two vertices in $U$ adjacent to a vertex in $C$; otherwise, there is a vertex $v\in V(C)$ such that there are three paths from $v$ to $U$ in $G$ that intersect each other only at $v$, and $U\cup\{v\}$ would be a larger 2-inseparable set.
A \emph{$U$-bridge} is a subgraph $H$ of $G$ such that $H$ is either an edge with both ends in $U$, or the subgraph obtained from a connected component $C$ of $G-U$ by adding the vertices $X\subseteq U$ adjacent to $C$ as well as all edges joining $X$ to $C$.
The \emph{attachments} of a $U$-bridge $H$ are the vertices in $U\cap V(H)$.
A \emph{3-block} of $G$ is a graph $B$ such that $V(B)$ is a maximal 2-inseparable set in $G$ and for all $u,v\in V(B)$, there is an edge $uv\in E(B)$ if and only if there exists a $V(B)$-bridge in $G$ whose attachments are $u$ and $v$. Note that, if $|V(B)|\geq 4$, then $B$ is a 3-connected graph.

\subsection{Group-labelled graphs}
Let $\G$ be an abelian group and let $H_1, H_2$ be subsets of $\G$. The \emph{set addition} of $H_1$ and $H_2$ is the set $H_1+H_2:=\{h_1+h_2:h_1\in H_1, h_2\in H_2\}$. If $H_1$ is a singleton, say $H_1=\{h_1\}$, then we may simply write $h_1+H_2$ to denote $H_1+H_2$. 
Let $\G'$ be a subgroup of $\G$. A \emph{coset} of $\G'$ is a set of the form $h+\G'$ for some $h\in \G$. 
We write $\G/\G'$ to denote the group of cosets of $\G'$ under set addition (called the \emph{quotient group} of $\G$ by $\G'$).

A \emph{$\G$-labelled graph} is a pair $(G,\g)$ of a graph $G$ and a labelling $\g:E(G)\to\G$.
We allow $G$ to have parallel edges as long as they have different labels. For a $\Gamma$-labelled graph~${(G,\g)}$ and a subgraph~${H}$ of~${G}$, we define~${\g(H):=\sum_{e \in E(H)} \g(e)}$ and call $\g(H)$ the \emph{$\g$-length} of $H$. If $\g(H)\neq 0$, then we say that $H$ is \emph{$\g$-nonzero}.

A \emph{3-block} of a $\G$-labelled graph $(G,\g)$ is a $\G$-labelled graph $(B,\g_B)$ such that $V(B)$ is a maximal 2-inseparable set of $G$ with $|V(B)|\geq 4$ and for all $u,v\in V(B)$, there is an edge $e\in E(B)$ with endpoints $u,v$ and with $\g_B(e)=g$ if and only if there is a $V(B)$-bridge $H$ of $G$ and a $u$-$v$-path $P$ in $H$ with $\g(P)=g$.

Let $\G'$ be a subgroup of $\G$.
We say that $(G,\g)$ is \emph{$\Gamma'$-balanced} if every cycle $C$ in $G$ with $|V(C)|\geq 3$ satisfies $\gamma(C)\in\Gamma'$. 
Note that every $\G$-labelled graph is $\G$-balanced.
If $(G,\g)$ is $\{0\}$-balanced, we simply say that it is \emph{balanced}.

\begin{lemma}[{\cite[Lemma 2.5]{thomas2023packingCycles}}] \label{lem:balanced3block}
    Let $\G$ be an abelian group, let $(G,\g)$ be a $\G$-labelled graph, let $(B,\g_B)$ be a 3-block of $(G,\g)$, and let $A\subseteq V(B)$ with $|A|\geq 3$. 
    If every $A$-path in $G$ has $\g$-length $0$, then $(B,\g_B)$ is balanced.
\end{lemma}

Suppose that $g$ is an element of $\G$ such that $2g=0$.
For a vertex $v$ in a $\G$-labelled graph $(G,\g)$, we say that the $\G$-labelled graph $(G,\g')$ where 
\begin{align*}
    \g'(e)=
    \left\{
	\begin{array}{ll}
		\g(e)+g  & \mbox{if $e$ is incident to $v$}\\
		\g(e) & \mbox{otherwise}
	\end{array}
\right.
\end{align*}
is obtained from $(G,\g)$ by \emph{shifting at $v$ by $g$}.
Observe that, since $2g=0$, $\g(H)=\g'(H)$ whenever $H$ is cycle or a path whose endpoints are not equal to $v$.
For two $\G$-labellings $\g,\g'$ of $G$, we say that $(G,\g)$ is \emph{shift-equivalent to} $(G,\g')$ if $(G,\g')$ can be obtained from $(G,\g)$ by a sequence of shifts.
We denote by $\bm{0}$ the $\G$-labelling that labels every edge 0.

\begin{lemma}
    [{\cite[Lemma 2.3]{thomas2023packingCycles}}] \label{lem:shiftequivalent}
    Let $\G$ be an abelian group and let $(B,\g_B)$ be a balanced $\G$-labelled graph such that $B$ is 3-connected.
    Then $(B,\g_B)$ is shift-equivalent to $(B,\bm{0})$.
\end{lemma}
Suppose that $(B,\g_B)$ is a 3-block of a $\G$-labelled graph $(G,\g)$. Let $v\in V(B)\subseteq V(G)$ and let $g\in \G$ with $2g=0$. Observe that if $(B,\g_B')$ and $(G,\g')$ are obtained from $(B,\g_B)$ and $(G,\g)$ respectively by shifting at $v$ by $g$, then $(B,\g_B')$ is a 3-block of $(G,\g')$.
Hence, Lemma \ref{lem:shiftequivalent} implies that, if $(B,\g_B)$ is balanced, then $(G,\g)$ is shift-equivalent to a $\G$-labelled graph $(G,\g')$ such that every $V(B)$-path in $G$ has $\g'$-length 0.

\begin{lemma}
    [{\cite[Theorem 1.1]{wollan2010packing}}]
    \label{lem:wollanApath}
    Let $\Gamma$ be an abelian group and let $(G,\gamma)$ be a $\G$-labelled graph with $A\subseteq V(G)$. Then for all positive integers $k$, either $G$ contains $k$ pairwise disjoint $\g$-nonzero $A$-paths or there exists $Z\subseteq V(G)$ with $|Z|<50k^4$ such that every $A$-path $P$ in $G-Z$ satisfies $\g(P)=0$. 
\end{lemma}

\subsection{Walls}
Let $c,r\geq 3$ be integers.
The \emph{elementary $(c,r)$-wall} is the graph $W_{c,r}$ obtained from the graph with vertex set $[2c]\times [r]$ and edge set 
\[\{(i,j)(i+1,j): i\in[2c-1], j\in[r]\}\cup\{(i,j)(i,j+1): i\in[2c],j\in[r-1], i+j\text{ is odd}\}\]
by deleting the two vertices of degree 1.
For $i\in[c]$ and $j\in[r]$, the \emph{$j$-th row} of $W_{c,r}$ is the path $W_{c,r}[\{(i,j)\in V(W_{c,r}):i\in[2c]\}]$ and the \emph{$i$-th column} is the path $W_{c,r}[\{(i',j)\in V(W_{c,r}): i'\in\{2i-1,2i\}, j\in[r]\}]$.

A \emph{$(c,r)$-wall} is a subdivision $W$ of the elementary $(c,r)$-wall $W'$. The vertices of $W$ before the subdivision, corresponding to vertices of $W'$, are called the \emph{nails} of $W$, and the set of nails of $W$ is denoted $N^W$.
Note that, given a wall $W$ as a graph, there may be multiple choices for the nails with degree 2; we assume that such a choice accompanies each wall.
The set of vertices of degree 3 in $W$ is denoted $\branch(W)$.
The \emph{$i$-th column} and \emph{$j$-th row} of $W$ are the paths in $W$ corresponding to the $i$-th column and $j$-th row of $W_{c,r}$, respectively, and are denoted $C_i^W$ and $R_j^W$, respectively.
The \emph{order} of an (elementary) $(c,r)$-wall is $\min\{c,r\}$, and a \emph{$c$-wall} is a wall of order at least $c$.

For $3\leq c'\leq c$ and $3\leq r'\leq r$, a \emph{$(c',r')$-subwall} of a $(c,r)$-wall $W$ is a $(c',r')$-wall $W'$ that is a subgraph of $W$ such that every row and column of $W'$ is a subpath of a row and column of $W$, respectively.
If, in addition, the indices $i$ and $j$ for which $C_i^W$ contains a column of $W'$ and $R_j^W$ contains a row of $W'$ form contiguous subsets of $[2c]$ and $[r]$ respectively, then we call $W'$ a \emph{compact} subwall of $W$. The choice of nails of a subwall $W'$ of $W$ are inherited from the nails of $W$.
We say that $W'$ is \emph{$k$-contained} in $W$ if $W'$ is disjoint from the first and last $k$ rows and columns of $W$.
A \emph{$c'$-column-slice} of $W$ is a $(c',r)$-subwall of $W$.

The \emph{boundary} of a $(c,r)$-wall $W$, denoted $\partial W$, is the union of the first and last columns and rows; that is, $\partial W = C_1^W\cup C_c^W \cup R_1^W\cup R_r^W$. We write $\partial N^W$ to denote the set of nails on the boundary, i.e.,~$\partial N^W = V(\partial W)\cap N^W$.
Note that $\partial W$ defines a natural cyclic ordering of the vertices in $\partial N^W$.
For $i\in\{1,c\}$ and $j\in\{1,r\}$, let $v_{i,j} \in \partial N^W$ denote the corner nail of $W$ in the intersection of $C_i^W\cap P_j^W$.
We write $\prec_W$ to denote the linear order on $\partial N^W$ that is a restriction of the cyclic ordering given by $\partial W$ such that $v_{1,1}$ is the minimum element and $v_{1,1}\prec_W v_{1,r}\prec_W v_{c,r}\prec_W v_{c,1}$.
The \emph{column-boundary} of $W$ is $V(C_1^W\cup C_c^W)\cap \partial N^W$.

Let $\G$ be an abelian group and let $\G'$ be a subgroup of $\G$.
A $\G$-labelled wall $(W,\g)$ (with a given choice of nails $N^W$) is \emph{strongly $\Gamma'$-balanced} if every $N^W$-path $P$ in $W$ satisfies $\g(P)\in\Gamma'$.
If $\G'=\{0\}$, we may simply say \emph{strongly balanced}.
Note that if $(W,\g)$ is $\G'$-balanced, then every 1-contained subwall of $(W,\g)$ is shift-equivalent to a strongly $\G'$-balanced wall by Lemma \ref{lem:shiftequivalent}.

Thomassen \cite{Thomassen1988} showed that for any fixed integer $m>0$, every wall of large enough order contains a large subwall $W'$ such that every $N^{W'}$-path in $W'$ has length 0 modulo $m$. This result has been generalized to finite abelian groups:

\begin{lemma}[{\cite[Lemma 5.3]{gollin2024unified}}]
    \label{lem:balancedsubwall}
    Let $\Gamma$ be a finite abelian group and let $w$ be a positive integer. Then there exists an integer $f_{\ref{lem:balancedsubwall}}(w,|\Gamma|)$ such that every  $\Gamma$-labelled wall $(W,\gamma)$ of order at least $f_{\ref{lem:balancedsubwall}}(w,|\Gamma|)$ contains a strongly balanced subwall $W'$ of order at least $w$.
\end{lemma}

\subsection{Handlebars}
Let $(X,\prec)$ be a linearly ordered set. For distinct $x_1,x_2,y_1,y_2\in X$ with $x_1\prec x_2$ and $y_1\prec y_2$, we say that $\{x_1,x_2\}$ and $\{y_1,y_2\}$ are
\begin{itemize}
    \item \emph{in series} if either $x_2\prec y_1$ or $y_2\prec x_1$;
    \item \emph{nested} if either $x_1\prec y_1\prec y_2\prec x_2$ or $y_1\prec x_1\prec x_2\prec y_2$; and
    \item \emph{crossing} otherwise.
\end{itemize}
A set $S\subseteq \binom{X}{2}$ of pairwise disjoint sets is \emph{in series, nested}, or \emph{crossing}, respectively, if its elements are pairwise in series, nested, or crossing, respectively; in these cases, we say that $S$ is \emph{pure} (with respect to $\prec$).

Let $W$ be a $(c,r)$-wall.
A \emph{$W$-handle} is a $W$-path whose endpoints are on the column-boundary of $W$. 
We say that a set $\mcp$ of $W$-handles is \emph{pure} if the set of the sets of endpoints of the paths in $\mcp$ is pure with respect to $\prec_W$.

The following lemma shows that a sufficiently large subset of $\binom{X}{2}$ contains a large pure subset, regardless of the size of $X$. In other words, any large enough set of disjoint $W$-handles contains a large subset of $W$-handles that is pure with respect to $\prec_W$.
\begin{lemma}
    [{\cite[Lemma 25]{huynh2019unified}}] \label{lem:purelinkages}
    Let $k$ be a positive integer. Let $(X,\prec)$ be a linearly ordered set and let $S\subseteq\binom{X}{2}$ be a set of pairwise disjoint sets with $|S|\geq (k-1)^3+1$.
    Then $S$ contains a pure subset of size at least $k$.
\end{lemma}

A \emph{$W$-handlebar} is a set $\mcp$ of disjoint $W$-handles that is pure and, in addition, there exist two paths $R,R'$ in $C_1^W\cup C_c^W$ such that each $W$-handle in $\mcp$ is a $V(R)$-$V(R')$-path.

A \emph{linkage} is a set of disjoint paths. For vertex sets~$A$ and~$B$, 
a set~$\mathcal{P}$ of disjoint $A$-$B$-paths is called an \emph{$A$-$B$-linkage}.
If $B$ is a subgraph of $G$, then an $A$-$V(B)$-linkage is also called an $A$-$B$-linkage.

Let $A$ be a vertex set disjoint from $V(W)$.
We say that an $A$-$W$-linkage $\mcq$ \emph{nicely links to $W$} if the endpoints of the paths in $\mcq$ in $W$ are vertices in $\partial N^W \cap V(R_1^W)$ of degree 2 in $W$.
An \emph{$A$-$W$-handle} is an $A$-$W$-path whose endpoint in $W$ is on the column-boundary of $W$. An \emph{$A$-$W$-handlebar} is a set $\mcq$ of $A$-$W$-handles such that the endpoints of the paths in $\mcq$ in $W$ are either all contained in $V(C_1^W)$ or all contained in $V(C_c^W)$.

Let $W'$ be a $c'$-column-slice of $W$ for some $3\leq c' \leq c$. If $P$ is a $W$-handle or an $A$-$W$-handle, the \emph{row-extension of $P$ to $W'$ in $W$} is a path $P'$ that is the $W'$-handle or the $A$-$W'$-handle, respectively, such that $P'$ contains $P$ and $P'$ is contained in the union of $P$ and the rows of $W$.
If $\mcp$ is a linkage whose paths are $W$-handles or $A$-$W$-handles, then the \emph{row-extension of $\mcp$ to $W'$ in $W$} is the set of row-extensions of the paths in $\mcp$ to $W'$ in $W$.

For each $i\in[2]$, let $\mcp_i$ be either a $W$-handlebar or an $A$-$W$-handlebar. If $\mcp_i$ is a $W$-handlebar, then there exist paths $R_i,R_i'$ in $C_1^W\cup C_c^W$ such that each path in $\mcp_i$ is a $V(R_i)$-$V(R_i')$-path; choose such $R_i,R_i'$ so that $V(R_i\cup R_i')$ is minimal, and define $U_i = V(R_i\cup R_i')$.
If $\mcp_i$ is an $A$-$W$-handlebar, then there exists a path $R_i$ in $C_1^W\cup C_c^W$ such that each path in $\mcp_i$ is an $A$-$V(R_i)$-path; choose such $R_i$ so that $V(R_i)$ is minimal, and define $U_i=V(R_i)$.
We say that $\mcp_1$ and $\mcp_2$ are \emph{non-mixing} if $U_1$ and $U_2$ are disjoint.

Let $\mcq_1,\mcq_2$ be $A$-$W$-handlebars such that either \begin{itemize}
    \item $\mcq_1=\mcq_2$ and $|\mcq_1|=2k$, or
    \item $\mcq_1$ and $\mcq_2$ are disjoint and non-mixing and $|\mcq_1|=|\mcq_2|=k$.
\end{itemize} 
A linkage obtained by \emph{joining $\mcq_1$ and $\mcq_2$} is a linkage $\mcp=\{P_1,\dots,P_k\}$ obtained by taking an enumeration $Q_1,\dots,Q_{2k}$ of $\mcq_1\cup\mcq_2$, letting $a_i$ denote the endpoint of $Q_i$ in $A$ for each $i\in[2k]$, and letting $P_i$ be obtained from $Q_{2i-1}$ and $Q_{2i}$ by identifying $a_{2i-1}$ and $a_{2i}$.
Observe that $\mcp$ can be made into either a nested or a crossing $W$-handlebar by an appropriate choice of the enumeration $Q_1,\dots,Q_{2k}$. 
If the endpoints of the paths in $\mcq_1\cup \mcq_2$ are all contained in the same column of $W$ (in particular, if $\mcq_1=\mcq_2$), then $\mcp$ can also be made into a $W$-handlebar in series.
If $\mcp$ is obtained by joining $\mcq_1$ and $\mcq_2$, we say that $\mcq_1$ and $\mcq_2$ are obtained by \emph{splitting} $\mcp$.

We conclude this subsection with a few technical lemmas from \cite{gollin2024unified, gollin2022unified} on dealing with $W$-handlebars in our proof of Theorem \ref{thm:mainApathstructure}. 
The first lemma shows that, given a balanced wall $W$, we can iteratively find disjoint sets of $\gamma$-nonzero $W$-handles.

\begin{lemma}[{\cite[Lemma 6.1]{gollin2024unified}}]
    \label{lem:addlinkage}
    There exist functions~${w_{\ref{lem:addlinkage}}\colon \mathbb{N}^2 \to \mathbb{N}}$ and~$f_{\ref{lem:addlinkage}} \colon \mathbb{N} \to \mathbb{N}$ satisfying the following. 
    Let~$k$, $t$, and~$c$ be positive integers with~${c \geq 3}$, let~${\Gamma}$ be an abelian group, and let~${(G, \gamma)}$ be a $\Gamma$-labelled graph.
    Let~$(W,\g)$ be a balanced wall in~$(G,\g)$ of order at least~${w_{\ref{lem:addlinkage}}(c,k)}$. 
    For each~${i \in [t-1]}$, let~$\mathcal{P}_i$ be a set of~$4k$ $W$-handles in~$G$ such that the paths in~${\bigcup_{i \in [t-1]} \mathcal{P}_i}$
    are pairwise disjoint. 
    If~$G$ contains at least~${f_{\ref{lem:addlinkage}}(k)}$ disjoint $\gamma$-nonzero ${\branch(W)}$-paths, 
    then there exist a $c$-column-slice~$W'$ of~$W$ and a set~$\mathcal{Q}_i$ of~$k$ pairwise disjoint $W'$-handles for each~${i \in [t]}$ such that 
    \begin{enumerate}
    [label=(\roman*)]
        \item for each~${i \in [t-1]}$, the set~$\mathcal{Q}_i$ is a subset of the row-extension of~$\mathcal{P}_i$ to~$W'$ in~$W$, 
        \item the paths in~${\bigcup_{i \in [t]} \mathcal{Q}_i}$ are pairwise disjoint, and
        \item the paths in~$\mathcal{Q}_t$ are $\gamma$-nonzero. 
    \end{enumerate}
\end{lemma}

The next lemma shows that, given a disjoint family $(\mcp_i:i\in[t])$ where each $\mcp_i$ is a set of disjoint $W$-handles, we can delete a few paths from each $\mcp_i$ to obtain non-mixing $W$-handlebars.

\begin{lemma}[{\cite[Lemma 5.1]{gollin2022unified}}]
    \label{lem:handlebarsnonmixing}
    There is a function ${f_{\ref{lem:handlebarsnonmixing}} \colon \mathbb{N}^2 \to \mathbb{N}}$ satisfying the following property. 
    Let $t$, $\theta$, $c$, and~$r$ be positive integers with~${r \geq 3}$ and~${c \geq 3}$, let~$W$ be a ${(c,r)}$-wall, and 
    let~${(\mathcal{P}_i \colon i \in [t])}$ be a family of pairwise disjoint sets of~${f_{\ref{lem:handlebarsnonmixing}}(t,\theta)}$ $W$-handles. 
    If the $W$-handles in $\bigcup_{i=1}^t \mathcal{P}_i$ are pairwise disjoint, 
    then there exists a family~${( \mathcal{P}^\ast_i \colon i \in [t] )}$ of pairwise non-mixing $W$-handlebars such that~${\mathcal{P}^\ast_i \subseteq \mathcal{P}_i}$ and~${\abs{\mathcal{P}^\ast_i} \geq \theta}$ for all~${i \in [t]}$. 
\end{lemma}

Given a disjoint family $(\mcp_i:i\in[t])$ of non-mixing $W$-handlebars, it will be convenient (for the purpose of satisfying length constraints) to combine the paths in $\mcp_i$ into longer paths, each containing exactly $d_i$ paths of $\mcp_i$ for some fixed $d_i$. The following lemma shows that this can be done along the first and last columns of $W$, and describes the types of the resulting $W$-handlebars.

\begin{lemma}[{\cite[Lemma 5.2]{gollin2022unified}}]
    \label{lem:combining-handles} 
    Let~$t$, $c$, $r$, and $\theta$ be positive integers with ${c \geq 5}$ and ${r \geq 3}$, and let $d_i$ be a positive integer for each~${i \in [t]}$. 
    Let~$W$ be a ${(c,r)}$-wall in a graph~$G$ and let $W^\ast$ be a $(c-2)$-column-slice of~$W$ containing~$C^W_2$ and~$C^W_{c-1}$. 
    Let~${(\mathcal{P}_i \colon i \in [t])}$ be a family of pairwise disjoint non-mixing $W$-handlebars with~${\abs{\mathcal{P}_i} \geq \theta d_i}$ 
    for all~${i \in [t]}$. 
    Then
    there is a family~${(\mathcal{P}^\ast_i \colon i \in [t])}$ of pairwise disjoint non-mixing $W^\ast$-handlebars each of size~$\theta$ such that
    for each~${i \in [t]}$, every path $Q$ in $\mcp_i^*$ contains exactly $d_i$ paths in $\mcp_i$ and none in $\mcp_j$ for $j\neq i$.
    Moreover, for each~${i \in [t]}$, if~$d_i$ is even, then~$\mathcal{P}^\ast_i$ is in series and if~$d_i$ is odd, 
    then~$\mathcal{P}^\ast_i$ is of the same type as~$\mathcal{P}_i$. 
\end{lemma}

\subsection{Tangles}
A \emph{separation} of a graph $G$ is an ordered pair $(C,D)$ of subsets of $V(G)$ with $V(G)=C\cup D$ such that $G[C]\cup G[D]=G$.
Its \emph{order} is defined to be $|C\cap D|$.
A \emph{tangle of order $t$} of $G$ is a set $\mct$ of separations of $G$ of order at most $t$ such that
\begin{itemize}
    \item if $(C,D)$ is a separation of $G$ of order less than $t$, then $\mct$ contains exactly one of $(C,D)$ and $(D,C)$, and
    \item if $(C_1,D_1),(C_2,D_2),(C_3,D_3)\in\mct$, then $G[C_1]\cup G[C_2]\cup G[C_3]\neq G$.
\end{itemize}
The two \emph{sides} of a separation $(C,D)$ are $C$ and $D$. If $(C,D)\in \mct$, then $C$ is the \emph{$\mct$-small} side and $D$ is the \emph{$\mct$-large} side.

Let $W$ be a wall of order $w$ in a graph $G$. Then the set $\mct_W$ of all separations $(C,D)$ of $G$ of order less than $w$ such that  $G[D]$ contains a row of $W$ is a tangle of order $w$. We say that $\mct_W$ is the tangle \emph{induced} by $W$.
If $\mct$ is a tangle such that $\mct_W\subseteq \mct$, then we say that $\mct$ \emph{dominates} the wall $W$.

\begin{theorem}
    [\cite{RST1994}]\label{thm:tanglewall}
    There exists a function $f_{\ref{thm:tanglewall}}:\mbn\to\mbn$ such that, for all integers $w\ge 3$, if a graph $G$ admits a tangle $\mct$ of order at least $f_{\ref{thm:tanglewall}}(w)$, then there is a $w$-wall $W$ in $G$ such that $\mct$ dominates $W$.
\end{theorem}

Let $\mct$ be a tangle of order $t\geq 3$ in a $\G$-labelled graph $(G,\g)$, and let $X$ be a subset of $V(G)$ with $|X|\leq t-3$. It is shown in \cite[Lemma 2.9]{thomas2023packingCycles} that there is a unique 3-block $(B,\g_B)$ of $(G-X,\g)$ such that $V(B)\cup X$ is not contained in any $\mct$-small side; we call $(B,\g_B)$ the \emph{$\mct$-large} 3-block of $(G-X,\g)$.

Let $W$ be a wall of order $w\geq 3$ in a $\Gamma$-labelled graph $(G,\g)$. Then $\branch(W)$ is 2-inseparable and, by the definition of $\mct_W$, $\branch(W)$ is not contained in any $\mct_W$-small side. In particular, there is no separation of order 2 separating $\branch(W)$ from the vertex set $V(B)$ of the the $\mct_W$-large 3-block $(B,\g_B)$ of $(G,\g)$, hence $\branch(W)\subseteq V(B)$. Moreover, if $W'$ is a 1-contained subwall of $W$, then $N^{W'}\subseteq \branch(W)$, so we have $N^{W'}\subseteq V(B)$.

Let $\mcf$ be a family of ($\G$-labelled) $A$-paths and let $(G,\g)$ be a $\G$-labelled graph.
For a function $f:\mbn\to\mbn$ and positive integer $k$, we say that $((G,\g),k)$ is a \emph{minimal counterexample to $f$ being an Erd\H{o}s-P\'osa function for $\mcf$} if $(G,\g)$ does not contain $k$ disjoint $A$-paths in $\mcf$ and there does not exist a set of at most $f(k)$ vertices intersecting every $A$-path in $\mcf$ in $G$, and subject to this, $k$ is minimum.
Standard techniques (e.g., \cite{bruhn2022packing, huynh2019unified, robertson1986graph, Thomassen1988}) show that such minimal counterexamples admit a large order tangle $\mct$ such that no member of $\mcf$ is contained in a $\mct$-small side.
We use the following lemma which was proved in \cite{bruhn2022packing} for the family of $A$-paths of length 0 modulo 4, but its proof applies verbatim to an arbitrary family $\mcf$ of $A$-paths.

\begin{lemma}
    [Lemma 8 in \cite{bruhn2022packing}] \label{lem:minctextangle}
    Let $\G$ be an abelian group and let $\mcf$ be a family of $\G$-labelled $A$-paths.
    Let $\theta$ be a positive integer and let $f:\mbn\to\mbn$ be a function such that $f(k)\geq 2f(k-1)+3\theta+10$.
    Suppose that $((G,\g),k)$ is a minimal counterexample to $f$ being an Erd\H{o}s-P\'osa function for $\mcf$. Then $G-A$ admits a tangle $\mct$ of order $\theta$ such that for each $(C,D)\in\mct$, $G[A\cup C]$ does not contain an $A$-path in $\mcf$.
\end{lemma}

\section{Obstructions} \label{sec:obstructions}

In this section we describe the obstructions to the \ep property of $A$-paths in group-labelled graphs. The obstructions consist of a wall and certain handlebars attached to the wall; they are adapted from similar constructions for cycles used in \cite{gollin2024unified, gollin2022unified}. We begin by summarizing these obstructions for cycles, which we call here \emph{ribboned walls}.

\subsection{Obstructions for cycles} \label{subsec:obscycles}
Let $k$ be a positive integer. A \emph{$k$-ribboned wall} is a tuple $\mcr=(W,(\mcp_i:i\in[m]))$ consisting of a wall $W$ of order at least $km$ and a family $(\mathcal{P}_i \colon i \in [m])$ of pairwise disjoint non-mixing $W$-handlebars each of size at least $k$. We write $G(\mcr):=W\cup \mcp_1\cup\dots\cup\mcp_m$.
A cycle $C$ in $G(\mcr)$ is \emph{minimally allowable} (with respect to $\mcr$) if $C$ contains exactly one path in $\mcp_i$ for each $i\in[m]$.

It is shown implicitly in the proofs of {\cite[Theorem 1.1]{gollin2024unified}} and {\cite[Theorem 3.3]{gollin2022unified}} that a $k$-ribboned wall contains a half-integral packing of $k$ minimally allowable cycles:
\begin{lemma}
[{\cite[Theorem 1.1]{gollin2024unified}}, {\cite[Theorem 3.3]{gollin2022unified}}] \label{lem:cyclehalfpacking}
    Let $\mcr$ be a $k$-ribboned wall.
    Then $G(\mcr)$ contains a half-integral packing of $k$ minimally allowable cycles.
\end{lemma}

A $k$-ribboned wall $\mcr=(W,(\mcp_i:i\in[m]))$ may not contain a large (integral) packing of minimally allowable cycles.
Let us say that $\mcr$ is an \emph{obstruction} (c.f.~\cite[Definition 3.2(O6)]{gollin2024unified}) if
\begin{enumerate}[label=(O)]
    \item \label{item:obstructions-handlebars}
        at least one of the following properties holds: 
        \begin{itemize}
            \item \label{subitem:obstructions-oddcrossing} 
            The number of crossing $W$-handlebars in~${( \mathcal{P}_i \colon i \in [m] )}$ is odd. 
            \item \label{subitem:obstructions-seriesnonseries} At least one but not all $W$-handlebars in~${( \mathcal{P}_i \colon i \in [m] )}$ are in series. 
            \item \label{subitem:obstructions-3series} At least three $W$-handlebars in~${( \mathcal{P}_i \colon i \in [m] )}$ are in series. 
        \end{itemize}
\end{enumerate}

It is shown in {\cite[Proposition 8.1]{gollin2022unified}} that an obstruction does not contain three disjoint minimally allowable cycles (see \cite[Figure 1]{gollin2022unified}). In fact, its proof shows the following:
\begin{lemma}[{\cite[Proposition 8.1]{gollin2022unified}}] \label{lem:cyclenopacking}
    Let $\mcr=(W,(\mcp_i:i\in[m]))$ be an obstruction. 
    Then there do not exist three disjoint cycles $C_1,C_2,C_3$ in $G(\mcr)$ such that for each $j\in[3]$ and $i\in[m]$,
    \begin{enumerate}
        \item $C_j$ contains at least one path in $\mcp_i$, and
        \item if $\mcp_i$ is nested or crossing, then $C_j$ contains an odd number of paths in $\mcp_i$.
    \end{enumerate}
\end{lemma}
Conversely, if $\mcr$ is not an obstruction (that is, if none of $(\mcp_i:i\in[m])$ are in series and the number of crossing $W$-handlebars is even, or if $m\leq 2$ and all $W$-handlebars are in series), then $G(\mcr)$ contains a large packing of minimally allowable cycles:
\begin{lemma}[{\cite[Lemmas 5.3 and 5.4]{gollin2022unified}}] \label{lem:cyclepacking}
    Let $\mcr$ be a $k$-ribboned wall that does not satisfy \ref{item:obstructions-handlebars}. Then $G(\mcr)$ contains a packing of $k$ minimally allowable cycles.
\end{lemma}

\subsection{Obstructions for \texorpdfstring{$A$}{A}-paths} \label{subsec:obsApaths}
Let $\G$ be a finite abelian group and let $A$ be a vertex set. A \emph{$\G$-labelled $(A,k)$-ribboned $\theta$-wall} is a tuple $\mcr = (W,(\mcp_i:i\in[m]),\mcq_1,\mcq_2,\g)$ consisting of a $\theta$-wall $W$ with $\theta\geq k(m+2)$, a family $(\mcp_i:i\in[m])$ of $W$-handlebars, $A$-$W$-handlebars $\mcq_1,\mcq_2$, and a $\G$-labelling $\g$ of the graph $G(\mcr):=W\cup \mcp_1\cup\dots\cup\mcp_m\cup\mcq_1\cup\mcq_2$ such that
\begin{enumerate}
        [label=(A\arabic*)]
        \item \label{item:Aobstructions-union}
        $\mcq_1$ and $\mcq_2$ are either disjoint or equal, and
        $\mcp_1,\dots,\mcp_m,\mcq_1\cup \mcq_2$ are pairwise disjoint and non-mixing with respect to $W$; moreover, we have $|\mcp_i|\geq k$ for all $i\in[m]$ and $|\mcq_j|\geq 2k$ for all $j\in[2]$, and
        \item \label{item:Aobstructions-zerowall}
            $W$ is strongly balanced and there exist $g_1,\dots,g_m,h_1,h_2\in\G$ such that every path in $\mcp_i$ has $\g$-length $g_i$ for all $i\in[m]$ and every path in $\mcq_j$ has $\g$-length $h_j$ for all $j\in[2]$.
\end{enumerate}
We may omit the phrase ``$\G$-labelled'' when $\G$ is clear from context. We may also omit the parameter $\theta$ and simply call $\mcr$ an \emph{$(A,k)$-ribboned wall}, but it is always assumed that the order of the wall $W$ is at least $k(m+2)$. We call $G(\mcr)$ the \emph{graph of $\mcr$}.

Let $\Lambda\subseteq\G$. 
We say that an $(A,k)$-ribboned wall $\mcr=(W,(\mcp_i:i\in[m]),\mcq_1,\mcq_2,\g)$ is \emph{$\Lambda$-irreducible} if, in addition to \ref{item:Aobstructions-union} and \ref{item:Aobstructions-zerowall}, the following hold:
    \begin{enumerate}
        [label=(A\arabic*)]
        \setcounter{enumi}{2}
        \item \label{item:Aobstructions-allowabletransversals}$(\sum_{i\in[m]}g_i)+h_1+h_2 \in \Lambda$, 
        \item \label{item:Aobstructions-minimality}
            for each~${i \in [m]}$ and $b_1,b_2\in\{0,1,2\}$ such that $b_1+b_2=2$, we have~$(\gen{g_j:j\in[m]\setminus\{i\}}+b_1h_1+b_2h_2) \cap \Lambda = \emptyset$, and
        \item \label{item:Aobstructions-even}
            for all $b_1,b_2\in\{0,1,2\}$ such that $b_1+b_2=2$,
            if $(\sum_{i \in [m]} c_ig_i)+b_1h_1+b_2h_2 \in \Lambda$ for some $c_1,\dots,c_m\in\mathbb{Z}$, then for each ${i \in [m]}$, either
            $\mathcal{P}_i$ is in series
            or~$c_i$ is odd.
    \end{enumerate}    
    We say that an $A$-path in $G(\mcr)$ is \emph{minimally allowable (with respect to $\mcr$)} if it contains
    \begin{itemize}
        \item exactly one path in $\mcp_i$ for each $i\in[m]$,
        \item exactly one path in $\mcq_1$ and in $\mcq_2$  if $\mcq_1$ and $\mcq_2$ are disjoint, and 
        \item exactly two paths in $\mcq_1$ if $\mcq_1=\mcq_2$.
    \end{itemize}
    It follows from \ref{item:Aobstructions-allowabletransversals} that minimally allowable $A$-paths are $\Lambda$-allowable.
    On the other hand, if $R$ is a $\Lambda$-allowable $A$-path in $G(\mcr)$, then $R$ contains at least one path in each $\mcp_i$ by \ref{item:Aobstructions-minimality} and, moreover, if $\mcp_i$ is nested or crossing, then $R$ contains an odd number of paths in $\mcp_i$ by \ref{item:Aobstructions-even}.

    A \emph{$\Lambda$-obstruction} is a $\Lambda$-irreducible  $(A,k)$-ribboned wall $(W,(\mcp_i:i\in[m]),\mcq_1,\mcq_2,\g)$ (satisfying \ref{item:Aobstructions-union}-\ref{item:Aobstructions-even}) such that   
    \begin{enumerate}
        [label=(A\arabic*)]
        \setcounter{enumi}{5}
        \item \label{item:Aobstructions-handlebars}
            at least one $W$-handlebar in~${( \mathcal{P}_i \colon i \in [m] )}$ is in series, and
        \item \label{item:Aobstructions-samesideApath}
            if $m=1$, then for each $j\in[2]$, we have $(\langle g_1\rangle+2h_j)\cap\Lambda=\emptyset$.
    \end{enumerate}

    We are ready to state our main structure theorem which says that, if we cannot find a large packing of $\Lambda$-allowable $A$-paths nor a small vertex set intersecting all $\Lambda$-allowable $A$-paths, then we can find a $\Lambda$-irreducible $(A,k)$-ribboned wall.
    \begin{restatable}{theorem}{mainApathstructure}
\label{thm:mainApathstructure}
        Let $\G$ be a finite abelian group and let $\Lambda\subseteq \G$. Then there exists a function $f:\mbn^2\to\mbn$ such that for every positive integer $k$ and $\G$-labelled graph $(G,\g)$ with $A\subseteq V(G)$, either
        \begin{enumerate}
            \item $(G,\g)$ contains $k$ disjoint $\Lambda$-allowable $A$-paths,
            \item there is a vertex set of size at most $f(k,|\G|)$ intersecting every $\Lambda$-allowable $A$-path, or
            \item $(G,\g)$ contains a $\Lambda$-irreducible $(A,k)$-ribboned wall.
        \end{enumerate}
    \end{restatable}
    Theorem \ref{thm:mainApathstructure} is proved in Sections \ref{sec:lemmas} and \ref{sec:proof}.
    In the remainder of this section, we prove Theorem \ref{thm:mainep} assuming Theorem \ref{thm:mainApathstructure}. 
    \begin{proposition}
        \label{prop:noEPC}
        Let $\G$ be a finite abelian group and let $\Lambda\subseteq\G$. If $\Lambda$ does not satisfy the Erd\H{o}s-P\'osa condition, then there exists a $\Lambda$-obstruction and the family of $\Lambda$-allowable $A$-paths does not satisfy the Erd\H{o}s-P\'osa property. 
    \end{proposition}
    \begin{proof}
         If $\Lambda$ does not satisfy the Erd\H{o}s-P\'osa condition, then $\Lambda$ fails either \ref{item:ep1} or \ref{item:ep2}. 
         
        First suppose that \ref{item:ep1} fails.
        Then there exist $a,b,c\in\G$ such that $a+b+c\in\Lambda$, $a+b\not\in\Lambda$, and $(2a+\langle c\rangle)\cap\Lambda=(2b+\langle c\rangle)\cap\Lambda=\emptyset$.
        Let $\mcr=(W,(\mcp_1),\mcq_1,\mcq_2,\g)$ be an $(A,k)$-ribboned wall such that every $N^W$-path in $W$ is $\g$-zero, $\mcp_1$ is in series, $\mcq_1$ and $\mcq_2$ are disjoint, and every path in $\mcp_1,\mcq_1,\mcq_2$ has $\g$-length $c,a,b$ respectively (see Figure \ref{fig:obs1}).\footnote{Note that the graphs in Figure \ref{fig:obstructions} are simplified versions and are not quite $(A,k)$-ribboned walls because they use grids instead of walls and, instead of $W$-handlebars (which should have disjoint paths attaching to the first or last column), we have paths attaching to the top row sharing endpoints.}
        
        Then \ref{item:Aobstructions-union} and \ref{item:Aobstructions-zerowall} are satisfied by definition and, since $a+b+c\in\Lambda$, \ref{item:Aobstructions-allowabletransversals} is satisfied.
        Since $a+b\not\in \Lambda$ and $(2a+\langle c\rangle)\cap\Lambda=(2b+\langle c\rangle)\cap\Lambda=\emptyset$, \ref{item:Aobstructions-minimality} and \ref{item:Aobstructions-samesideApath} are satisfied.
        Moreover, since $\mcp_1$ is the only $W$-handlebar and it is in series, \ref{item:Aobstructions-even} and \ref{item:Aobstructions-handlebars} are satisfied.
        Hence, $\mcr$ is a $\Lambda$-obstruction. 

        Since every $\Lambda$-allowable $A$-path contains one path in each $\mcq_1,\mcq_2$ and at least one path in $\mcp_1$, and since $G(\mcr)$ is planar, it is clear that no two $\Lambda$-allowable $A$-paths are disjoint. On the other hand, there does not exist a vertex set intersecting all $\Lambda$-allowable $A$-paths with fewer than $k$ vertices in $G(\mcr)$. Hence, the family of $\Lambda$-allowable $A$-paths does not satisfy the Erd\H{o}s-P\'osa property.

        Now suppose that \ref{item:ep2} fails. Then there exist $a,b,c\in\G$ such that $2a+b +  c\in\Lambda$ and $(2a+\langle b\rangle)\cap\Lambda=(2a+\langle c\rangle)\cap\Lambda=\emptyset$.
        Let $\mcr=(W,(\mcp_1,\mcp_2),\mcq_1,\mcq_2,\g)$ be an $(A,k)$-ribboned wall such that every $N^W$-path in $W$ is $\g$-zero, $\mcp_1,\mcp_2$ are in series, $\mcq_1=\mcq_2$, and every path in $\mcp_1,\mcp_2,\mcq_1$ has $\g$-length $b,c,a$ respectively (see Figure \ref{fig:obs2}).

        Then \ref{item:Aobstructions-union} and \ref{item:Aobstructions-zerowall} are satisfied by definition and, since $2a+b+c\in\Lambda$, \ref{item:Aobstructions-allowabletransversals} is satisfied. 
        Since $(2a+\langle b\rangle)\cap\Lambda = (2a+\langle c\rangle)\cap\Lambda = \emptyset$, \ref{item:Aobstructions-minimality} is satisfied.
        Since all handlebars $(\mcp_i:i\in[2])$ are in series, \ref{item:Aobstructions-even} and \ref{item:Aobstructions-handlebars} are satisfied. Since there are more than one handlebars, \ref{item:Aobstructions-samesideApath} is trivially satisfied.
        Hence, $\mcr$ is a $\Lambda$-obstruction. 

        Every $\Lambda$-allowable $A$-path contains two paths in $\mcq_1\cup\mcq_2$ and at least one path in each $\mcp_1,\mcp_2$. Since $\mcp_1,\mcp_2,(\mcq_1\cup\mcq_2)$ are pairwise non-mixing and since $G(\mcr)$ is planar, no two $\Lambda$-allowable $A$-paths in $G(\mcr)$ are disjoint, but there does not exist a vertex set intersecting all $\Lambda$-allowable $A$-paths with fewer than $k$ vertices in $G(\mcr)$. Hence, the family of $\Lambda$-allowable $A$-paths does not satisfy the Erd\H{o}s-P\'osa property.
    \end{proof}

    \begin{proposition} \label{prop:EPnoobstruction}
        Let $\G$ be a finite abelian group and let $\Lambda\subseteq \G$. If $\Lambda$ satisfies the Erd\H{o}s-P\'osa condition, then there does not exist a $\Lambda$-obstruction.
    \end{proposition}
    \begin{proof}
        Suppose that there is a $\Lambda$-obstruction $\mcr=(W,(\mcp_i:i\in[m]),\mcq_1,\mcq_2,\g)$, satisfying \ref{item:Aobstructions-union}-\ref{item:Aobstructions-samesideApath}.
        Note that by \ref{item:Aobstructions-handlebars}, $m\geq 1$ and at least one of $(\mcp_i:i\in[m])$ is in series.

        First consider the case that there is a $W$-handlebar in $(\mcp_i:i\in[m])$ that is either nested or crossing, say $\mcp_m$. Then $m\geq 2$ by \ref{item:Aobstructions-handlebars}.
        Define $a=h_1+g_m$, $b=h_2$, and $c=\sum_{i=1}^{m-1} g_i$.
        Then $a+b+c\in\Lambda$ by \ref{item:Aobstructions-allowabletransversals}, and we have $a+b\not\in\Lambda$ and $(2b+\langle c\rangle)\cap\Lambda=\emptyset$ by \ref{item:Aobstructions-minimality}. We also have $(2a+\langle c\rangle)\cap\Lambda=\emptyset$ by \ref{item:Aobstructions-even}. This violates \ref{item:ep1}, hence $\Lambda$ fails the Erd\H{o}s-P\'osa condition.

        So we may assume that every $W$-handlebar in $(\mcp_i:i\in[m])$ is in series.
        If $m=1$, define $a=h_1,b=h_2$, and $c=g_1$. Then $a+b+c\in\Lambda$ by \ref{item:Aobstructions-allowabletransversals} and $a+b\not\in\Lambda$ by \ref{item:Aobstructions-minimality}. Moreover, by \ref{item:Aobstructions-samesideApath}, we have $(2a+\langle c\rangle)\cap\Lambda = (2b+\langle c\rangle)\cap\Lambda = \emptyset$. This violates \ref{item:ep1}, hence $\Lambda$ fails the Erd\H{o}s-P\'osa condition.

        So we may assume that $m\geq 2$. Suppose further that $(2h_1+\langle g_1,\dots,g_m\rangle)\cap\Lambda\neq\emptyset$. Then there exist integers $c_1,\dots,c_m$ such that $2h_1+\sum_{i\in[m]}c_ig_i\in\Lambda$.
        Define $a=h_1$, $b=\sum_{i\in[m-1]}c_ig_i$, and $c=c_mg_m$. Then $2a+b+c\in \Lambda$, but $(2a+\langle b\rangle)\cap\Lambda = (2a+\langle c\rangle)\cap\Lambda = \emptyset$ by \ref{item:Aobstructions-minimality}. This violates \ref{item:ep2}, hence $\Lambda$ fails the Erd\H{o}s-P\'osa condition. Similarly, if $(2h_2+\langle g_1,\dots,g_m\rangle)\cap\Lambda\neq\emptyset$, then $\Lambda$ fails the Erd\H{o}s-P\'osa condition.

        So we may assume that $(2h_1+\langle g_1,\dots,g_m\rangle)\cap\Lambda = (2h_2+\langle g_1,\dots,g_m\rangle)\cap\Lambda = \emptyset$.
        Define $a=h_1,b=h_2$, and $c=\sum_{i\in[m]}g_i$.
        Then we have $a+b+c\in\Lambda$ by \ref{item:Aobstructions-allowabletransversals}, $a+b\not\in\Lambda$ by \ref{item:Aobstructions-minimality}, and $(2a+\langle c\rangle)\cap\Lambda = (2b+\langle c\rangle)\cap\Lambda=\emptyset$ by assumption. This violates \ref{item:ep1}, hence $\Lambda$ fails the Erd\H{o}s-P\'osa condition.
    \end{proof}

    \begin{proposition} \label{lem:Apathpacking}
        Let $\G$ be a finite abelian group and let $\Lambda\subseteq\G$.
        Let $\mcr=(W,(\mcp_i:i\in[m]),\mcq_1,\mcq_2,\g)$ be a $\Lambda$-irreducible $(A,k)$-ribboned wall. Then $G(\mcr)$ contains a half-integral packing of $k$ $\Lambda$-allowable $A$-paths.
        Moreover, if $\Lambda$ satisfies the Erd\H{o}s-P\'osa condition, then $G(\mcr)$ contains a packing of $k$ $\Lambda$-allowable $A$-paths.
    \end{proposition}
    \begin{proof}
        Let $\mcp_{m+1}$ be the $W$-handlebar obtained by joining $\mcq_1$ and $\mcq_2$ so that
        \begin{itemize}
            \item $\mcp_{m+1}$ is nested, if the number of crossing $W$-handlehars in $(\mcp_i:i\in[m])$ is even, and
            \item $\mcp_{m+1}$ is crossing, if the number of crossing $W$-handlebars in $(\mcp_i:i\in[m])$ is odd.
        \end{itemize}
        Then $\mcr^\circ:=(W,(\mcp_i:i\in[m+1]))$ is a $k$-ribboned wall.
        By Lemma \ref{lem:cyclehalfpacking}, $G(\mcr^\circ)$ contains a half-integral packing of $k$ minimally allowable cycles with respect to $\mcr^\circ$. By splitting $\mcp_{m+1}$, we obtain a half-integral packing of $k$ $A$-paths in $G(\mcr)$ that are minimally allowable with respect to $\mcr$ and hence $\Lambda$-allowable.

        Moreover, suppose that $\Lambda$ satisfies the Erd\H{o}s-P\'osa condition. We show that there is a packing of $k$ $\Lambda$-allowable $A$-paths in $G(\mcr)$.
        By Proposition \ref{prop:EPnoobstruction}, there does not exist a $\Lambda$-obstruction, so $\mcr$ fails to satisfy either \ref{item:Aobstructions-handlebars} or \ref{item:Aobstructions-samesideApath}.

        First suppose that \ref{item:Aobstructions-handlebars} is not satisfied. By the definition of $\mcp_{m+1}$, no $W$-handlebar in $(\mcp_i:i\in[m+1])$ is in series and the number of crossing $W$-handlebars in $(\mcp_i:i\in[m+1])$ is even. So $\mcr^\circ$ does not satisfy \ref{item:obstructions-handlebars} and, by Lemma \ref{lem:cyclepacking}, $G(\mcr^\circ)$ contains a packing of $k$ minimally allowable cycles.
        By splitting $\mcp_{m+1}$, we obtain a packing of $k$ $\Lambda$-allowable $A$-paths in $G(\mcr)$.

        So we may assume that \ref{item:Aobstructions-handlebars} holds and that \ref{item:Aobstructions-samesideApath} is not satisfied. Then $m=1$, so $\mcp_1$ is in series by \ref{item:Aobstructions-handlebars}, and since \ref{item:Aobstructions-samesideApath} fails, there exists $j\in[2]$ such that $(\langle g_1\rangle+2h_j)\cap \Lambda \neq \emptyset$, where $g_1,h_j$ are as in \ref{item:Aobstructions-zerowall}. Let us assume without loss of generality that $(\langle g_1\rangle+2h_2)\cap\Lambda\neq\emptyset$. 

        Let $\mcp_2'$ be the $W$-handlebar obtained by joining $\mcq_2$ and $\mcq_2$ so that $\mcp_2'$ is in series.
        Then $\mcr'=(W,(\mcp_1,\mcp_2'))$ is a $k$-ribboned wall with two $W$-handlebars, both of which are in series. So $\mcr'$ does not satisfy \ref{subitem:obstructions-3series} and, by Lemma \ref{lem:cyclepacking}, $G(\mcr')$ contains a packing of $k$ minimally allowable cycles. By splitting $\mcp_2'$, we obtain a packing of $k$ $\Lambda$-allowable $A$-paths in $G(\mcr)$.
    \end{proof}

    Assuming Theorem \ref{thm:mainApathstructure}, these propositions can be summarized into the following theorem, which immediately implies Theorem \ref{thm:mainep}.
    \begin{theorem}
        Let $\G$ be a finite abelian group and let $\Lambda\subseteq \G$. Let $A$ be a vertex set. Then 
        \begin{enumerate}
            \item the family of $\Lambda$-allowable $A$-paths satisfies the half-integral Erd\H{o}s-P\'osa property, and \label{item:main1}
            \item the following are equivalent:
            \begin{enumerate}[label=(2\alph*)]
                \item The family of $\Lambda$-allowable $A$-paths satisfies the Erd\H{o}s-P\'osa property. \label{item:main2a}
                \item $\Lambda$ satisfies the Erd\H{o}s-P\'osa condition. \label{item:main2b}
                \item There does not exist a $\Lambda$-obstruction. \label{item:main2c}
            \end{enumerate} 
        \end{enumerate}
    \end{theorem}
    \begin{proof}
        Let $(G,\g)$ be a $\G$-labelled graph.
        By Theorem~\ref{thm:mainApathstructure}, either 
        \begin{itemize}
            \item $(G,\g)$ contains $k$ disjoint $\Lambda$-allowable $A$-paths,
            \item there is a vertex set of size at most $f(k)$ intersecting all $\Lambda$-allowable $A$-paths, or
            \item $(G,\g)$ contains a $\Lambda$-irreducible $(A,k)$-ribboned wall $\mcr$. 
        \end{itemize}  
        In the first two outcomes, we are done.
        In the third outcome, by Proposition \ref{lem:Apathpacking}, $G(\mcr)$ contains a half-integral packing of $k$ $\Lambda$-allowable $A$-paths. This proves \eqref{item:main1}. 

        Moreover, by Proposition \ref{lem:Apathpacking}, if $\Lambda$ satisfies the Erd\H{o}s-P\'osa condition, then $G(\mcr)$ contains a packing of $k$ $\Lambda$-allowable $A$-paths. This proves that \ref{item:main2b} implies \ref{item:main2a}. 
        Proposition \ref{prop:noEPC} shows that \ref{item:main2a} implies \ref{item:main2b} and that \ref{item:main2c} implies \ref{item:main2b}.
        Proposition \ref{prop:EPnoobstruction} shows that \ref{item:main2b} implies \ref{item:main2c}.
    \end{proof}

Let us briefly discuss how the proof can be adapted to accommodate long paths (paths with at least $\ell$ edges, for fixed $\ell$). In Lemmas \ref{lem:cyclehalfpacking} and \ref{lem:cyclepacking}, by taking a $k$-ribboned wall $\mcr=(W,(\mcp_i:i\in[m]))$ such that $W$ has a slightly larger order, say at least $km+k\ell$, one can ensure that, in the (half-integral) packing of minimally allowable cycles, each cycle has at least $\ell$ edges.
Hence, in Proposition \ref{lem:Apathpacking}, by taking a $\Lambda$-irreducible $(A,k)$-ribboned wall whose wall has a slightly larger order, we can obtain a (half-integral) packing of long $\Lambda$-allowable $A$-paths. So we just need to ask for a $\Lambda$-irreducible $(A,k)$-ribboned wall whose wall has a slightly larger order in Theorem \ref{thm:mainApathstructure}; this can easily be done by starting with a larger wall in the proof of Theorem \ref{thm:mainApathstructure}.

\section{Proof outline and lemmas} \label{sec:lemmas}
Let us first outline the proof of Theorem \ref{thm:mainApathstructure} along with a rough description of lemmas proved in this section.
Let $\Gamma$ be a finite abelian group and let $\Lambda\subseteq \G$. 
If a $\G$-labelled graph $(G,\g)$ does not contain a large packing of $\Lambda$-allowable $A$-paths nor a small vertex set intersecting every $\Lambda$-allowable $A$-path, then $G-A$ contains a large order tangle as described in Lemma \ref{lem:minctextangle}, from which we obtain a large strongly balanced wall $W$ by Lemmas \ref{thm:tanglewall} and \ref{lem:balancedsubwall}.

We first construct a large ribboned wall by iteratively applying Lemma \ref{lem:addlinkage} to obtain $W$-handlebars whose paths are $\g$-nonzero, as long as the $\g$-lengths of the new paths allow us to generate more elements of $\G$ than the current $W$-handlebars. At the end of this process, we obtain a subgroup $\G'$ of $\G$ such that the $\g$-lengths of the obtained $W$-handlebars generate $\G'$ (we call such a ribboned wall \emph{$\G'$-generating}) and, moreover, all $W$-handles whose $\g$-lengths are not in $\G'$ can be destroyed by a small number of vertices. This is done in Lemma \ref{lem:generatinglinkagessubwall} to obtain a $\G'$-generating ribboned wall $(W',(\mcp_1,\dots,\mcp_m))$.

From this, we construct an $(A,k)$-ribboned wall by adding $A$-$W'$-handlebars of certain $\g$-lengths. We first show that there is a pair of cosets $H_1,H_2$ of $\G'$ such that $H_1+H_2$ contains an element of $\Lambda$ and such that there does not exist a small vertex set intersecting all $A$-$W'$-paths whose $\g$-length is in $H_1$ and there does not exist a small vertex set intersecting all $A$-$W'$-paths whose $\g$-length is in $H_2$. 
This allows us to find a large subwall $W''$ of $W'$ and two large $A$-$W''$-linkages $\mcq_1,\mcq_2$ that nicely link to $W''$ such that every path in $\mcq_j$ has $\g$-length in the coset $H_j$, for $j\in[2]$. This is done in Lemma \ref{lem:H1H2linkages} by iteratively applying Lemma \ref{lem:AWpathslengthL}, which is a lemma from \cite{thomas2023packingApaths} stating that for each $\ell\in \G$, either there is a subwall $W''$ of $W'$ and a large $A$-$W''$-linkage that nicely links to $W''$ whose paths have $\g$-length $\ell$, or there is a small vertex set intersecting all such paths.

At this point, the linkages $\mcq_1,\mcq_2$ need not be disjoint from each other or from $\mcp_1,\dots,\mcp_m$. 
In Lemma \ref{lem:PQdisjointwithsame12attachments}, we show that a new $A$-$W''$-linkage can be made disjoint from an existing set of handlebars by possibly losing a few paths in each handlebar.
In Lemma \ref{lem:generatingwallribbon}, we apply Lemma \ref{lem:PQdisjointwithsame12attachments} twice, once to add $\mcq_1$ to $\mcp_1\cup\dots\cup\mcp_m$, then a second time to add $\mcq_2$; moreover, we show that these handlebars can be made non-mixing by applying Lemma \ref{lem:handlebarsnonmixing}. This allows us to obtain an $(A,k')$-ribboned wall for some large $k'$. Finally, we show in Lemma \ref{lem:makeirreducible} that we can reduce this to a $\Lambda$-irreducible $(A,k)$-ribboned wall.

Let $W$ be a wall and for a non-negative integer $m$, let $\mcp_1,\dots,\mcp_m$ be $W$-handlebars in a $\Gamma$-labelled graph $(G,\gamma)$.
For a subgroup $\Gamma'\le\Gamma$, we say that $(W,\{\mcp_1,\dots,\mcp_m\})$ is \emph{$\Gamma'$-generating} if $W$ is a wall of order at least $\sum_{i\in[m]}|\mcp_i|$ that is strongly balanced, $\mcp_1,\dots,\mcp_m$ are pairwise disjoint and non-mixing, and there exist $g_1,\dots,g_m\in\Gamma'-\{0\}$ such that $\Gamma'=\langle g_1,\dots,g_m\rangle$ (if $m=0$, then $\G'=\{0\}$) and, for all $i\in[m]$, we have that every path in $\mcp_i$ has $\g$-length $g_i$ and that $\G'\neq\langle \{g_1,\dots,g_m\}\setminus\{g_i\}\rangle$. Note that $m<|\G'|$.

\begin{lemma}
    \label{lem:generatinglinkagessubwall}
    There exist functions $f_{\ref{lem:generatinglinkagessubwall}}:\mbn^3\to\mbn$ and $g_{\ref{lem:generatinglinkagessubwall}}:\mbn^2\to\mbn$ satisfying the following.
    Let $\G$ be a finite abelian group and let $w,k$ be positive integers with $w\geq 3$.
    If $(G,\g)$ is a $\G$-labelled graph containing a $f_{\ref{lem:generatinglinkagessubwall}}(w,k,|\G|)$-wall $W$, then there exists a subgroup $\G'\le\G$ and a vertex set $Z\subseteq V(G)$ with $|Z|\le g_{\ref{lem:generatinglinkagessubwall}}(k,|\G|)$ such that $(G-Z,\g)$ contains a strongly balanced $w$-subwall $W'$ of $W$ and a family of $W'$-handlebars $\mcp_1,\dots,\mcp_m$ each of size at least $k$ such that
    \begin{enumerate}
    [label=(\roman*)]
        \item $(W',\{\mcp_1,\dots,\mcp_m\})$ is $\Gamma'$-generating, and
        \item the $\mct_{W'}$-large 3-block $(B,\g_B)$ of $(G-Z,\gamma)$ is $\Gamma'$-balanced.
    \end{enumerate}
\end{lemma}
\begin{proof}
    Define 
    \begin{align*}
    g_{\ref{lem:generatinglinkagessubwall}}(k,|\G|) &=50|\G|f_{\ref{lem:addlinkage}}(k|\G|^{|\G|})^4, \\
    w_1 &= w_{\ref{lem:addlinkage}}(w,k|\G|), \\
    w_i &= w_{\ref{lem:addlinkage}}(w_{i-1},k(4|\G|)^i) \text{ for }i\geq 2,\text{ and}\\f_{\ref{lem:generatinglinkagessubwall}}(w,k,|\G|) &= f_{\ref{lem:balancedsubwall}}(w_{|\G|},|\G|)+g_{\ref{lem:generatinglinkagessubwall}}(k,|\G|).
    \end{align*}
    Suppose that $(W,\g)$ is an $f_{\ref{lem:generatinglinkagessubwall}}(w,k,|\G|)$-wall in $(G,\g)$.
    Choose a subgroup $\G''$ of $\G$ and a vertex set $Z\subseteq V(G)$ with $|Z|\le 50(|\G|-|\G''|)f_{\ref{lem:addlinkage}}(k|\G|^{|\G|})^4$ such that the $\mct_W$-large 3-block $(B,\g_B)$ of $(G-Z,\g)$ is $\G''$-balanced and, subject to this condition, $\G''$ is minimal. Note that this condition is satisfied by $\G$ and $\emptyset$  because $(G,\g)$ (and hence the $\mct_W$-large 3-block of $(G-\emptyset,\g)$) is $\G$-balanced. Hence, such a choice of $\G''$ and $Z$ exists.

    Next, choose an integer $m$, a strongly balanced $w_{|\G|-m}$-subwall $W'$ of $W$ in $(G-Z,\g)$, and $W'$-handlebars $\mcp_1,\dots,\mcp_m$ in $(G-Z,\g)$, each of size at least $k(4|\G|)^{|\G|-m}$, so that $(W',\{\mcp_1,\dots,\mcp_m\})$ is $\G'$-generating for some $\G'\le\G$ and, subject to this condition, $m$ is maximal.  
    Again, note that such a choice exists; since 
    \[|Z|\le 50(|\G|-|\G''|)f_{\ref{lem:addlinkage}}(k|\G|^{|\G|})^4\le g_{\ref{lem:generatinglinkagessubwall}}(k,|\G|),\] 
    and $W$ has order \[f_{\ref{lem:generatinglinkagessubwall}}(w,k,|\G|) = f_{\ref{lem:balancedsubwall}}(w_{|\G|},|\G|)+g_{\ref{lem:generatinglinkagessubwall}}(k,|\G|),\] 
    there is a $f_{\ref{lem:balancedsubwall}}(w_{|\G|},|\G|)$-subwall $W^*$ of $W$ in $(G-Z,\g)$, and by Lemma \ref{lem:balancedsubwall}, there is a strongly balanced $w_{|\G|}$-subwall $W^{**}$ of $W^*$ (and hence of $W$) , so $(W^{**},\emptyset)$ is $\{0\}$-generating.

    We claim that $\G'=\G''$. Note that we have $\G'\le\G''$ because the $\mct_W$-large 3-block of $(G-Z,\g)$ is $\G''$-balanced and $W'\subseteq G-Z$ is a subwall of $W$ (this implies that $\mct_{W'}\subseteq \mct_W$ and hence the $\mct_{W'}$-large 3-block of $(G-Z,\g)$ is also $(B,\g_B)$).
    Now suppose that $\G'$ is a proper subgroup of $\G''$.

    By Lemma \ref{lem:wollanApath} (applied to $\G/\G'$), either there exist $f_{\ref{lem:addlinkage}}(k(4|\G|)^{|\G|-m})$ disjoint $V_{\neq 2}(W')$-paths in $(G-Z,\g)$ whose $\g$-lengths are not in $\G'$, or there is a vertex set $Z'\subseteq V(G)-Z$ with $|Z'|<50f_{\ref{lem:addlinkage}}(k(4|\G|)^{|\G|-m})^4$ such that every $V_{\neq 2}(W')$-path in $(G-Z-Z',\g)$ has $\g$-length in $\G'$.

    In the first case, since $W'$ is a wall of order \[w_{|\G|-m} = w_{\ref{lem:addlinkage}}\left( w_{|\G|-m-1}, k(4|\G|)^{|\G|-m}\right)\] and since each $\mcp_i$ has size at least $k(4|\G|)^{|\G|-m}$, we have by Lemma \ref{lem:addlinkage} (applied to $\G/\G')$ that there exist a $w_{|\G|-m-1}$-column-slice $W''$ of $W'$ and pairwise disjoint $W''$-handlebars $\mcp'_1,\dots,\mcp'_{m+1}$, each of size $\frac14 k(4|\G|)^{|\G|-m}$, such that $\mcp'_i$ is a subset of the row-extension of $\mcp_i$ to $W''$ in $W'$ for $i\in[m]$, and every path $P$ in $\mcp'_{m+1}$ satisfies $\g(P)\not\in \G'$.
    Then there is a subset $\mcp''_{m+1}\subseteq \mcp'_{m+1}$ with $|\mcp''_{m+1}|=k(4|\G|)^{|\G|-m-1}$ such that the paths in $\mcp''_{m+1}$ all have the same $\g$-length, say $g_{m+1}\not\in\G'$.
    Note that $W''$ is strongly balanced because it is a subwall of $W'$, which is strongly balanced.
    Hence $(W'',\{\mcp'_1,\dots,\mcp'_m,\mcp''_{m+1}\})$ is $\G'''$-generating, where $\G'''= \langle \Gamma'\cup\{g_{m+1}\}\rangle$, and this contradicts the maximality of $m$.

    So we may assume the second case, that there is a vertex set $Z'\subseteq V(G)-Z$ with $|Z'|<50f_{\ref{lem:addlinkage}}(k|\G|^{|\G|-m})^4$ such that every $V_{\neq 2}(W')$-path in $(G-Z-Z',\g)$ has $\g$-length in $\G'$.
    By Lemma \ref{lem:balanced3block} applied to $\G/\G'$ and $(G-Z-Z',\g)$, we have that the $\mct_{W'}$-large 3-block of $(G-Z-Z',\g)$ is $\G'$-balanced.
    Since $\G'$ is a proper subgroup of $\G''$, we have $|Z\cup Z'| < 50(|\G|-|\G''|)f_{\ref{lem:addlinkage}}(k|\G|^{|\G|})^4+ 50f_{\ref{lem:addlinkage}}(k|\G|^{|\G|})^4 \le 50(|\G|-|\G'|)f_{\ref{lem:addlinkage}}(k|\G|^{|\G|})^4$, contradicting the minimality of $\G''$.

    This proves our claim that $\G'=\G''$.
    Thus, we have a vertex set $Z\subseteq V(G)$ with $|Z|\leq 50(|\G|-|\G'|)f_{\ref{lem:addlinkage}}(k|\G|^{|\G|})^4 \leq 50|\G|f_{\ref{lem:addlinkage}}(k|\G|^{|\G|})^4=g_{\ref{lem:generatinglinkagessubwall}}(k,|\G|)$ such that $W'$ is a strongly balanced subwall of $W$ in $(G-Z,\g)$ of order $w_{|\G|-m}\geq w$, and $\mcp_1,\dots,\mcp_m$ is a family of $W'$-handlebars in $(G-Z,\g)$, each of size $k(4|\G|)^{|\G|-m} \geq k$, such that $(W',\{\mcp_1,\dots,\mcp_m\})$ is $\G'$-generating. Moreover, the $\mct_{W'}$-large 3-block of $(G-Z,\g)$ is $\G'$-balanced.
\end{proof}

\begin{lemma}
    [{\cite[Lemma 4.9]{thomas2023packingApaths}}]
    \label{lem:AWpathslengthL}
    Let $\Gamma$ be an abelian group with $\ell\in\Gamma$, let $w,t$ be positive integers with $w\geq 12t$, and let $T=3(36t)^2$.
    Let $(G,\gamma)$ be a $\Gamma$-labelled graph with $A\subseteq V(G)$ and let $(B,\gamma_B)$ be a 3-block of $(G-A,\gamma)$ such that $\g_B=\bm{0}$.
    Let $W'$ be an $(s+2)$-wall in $G-A$, where $s\geq (2w+1)(2T+1)$, and let $W$ be a 1-contained $s$-subwall of $W'$, such that $N^W\subseteq V(B)$.
    Suppose in addition that there does not exist $X\subseteq V(G)$ with $|X|<12T$ intersecting all $A$-$N^W$-paths of $\g$-length $\ell$ in $(G,\gamma)$.
    Then $W$ contains a compact $w$-subwall $W_1$ such that there are $t$ disjoint $A$-$W_1$-paths each of $\g$-length $\ell$ that nicely link to $W_1$.
\end{lemma}

\begin{lemma}\label{lem:H1H2linkages}
    There exists a function $\beta_{\ref{lem:H1H2linkages}}:\mbn^3\to\mbn$ satisfying the following.
    Let $\Gamma$ be a finite abelian group and let $\Lambda\subseteq\Gamma$.
    Let $w,t$ be positive integers such that $w\geq 12t$, and define $T=3(36t)^2$.
    If $(G,\g)$ is a $\G$-labelled graph with $A\subseteq V(G)$ such that
    \begin{enumerate}[label=(\Roman*)]
        \item there does not exist $Y\subseteq V(G)$ with $|Y|<12T|\G|$ intersecting every $\Lambda$-allowable $A$-path, \label{item:h1h2i}
        \item $(G-A,\g)$ contains a strongly balanced $\beta_{\ref{lem:H1H2linkages}}(w,t,|\G|)$-wall $W'$ inducing a tangle $\mct=\mct_{W'}$ in $G-A$ such that the $\mct$-large 3-block $(B,\g_B)$ of $(G-A,\g)$ is a $\G'$-labelled graph for some subgroup $\G'\le\G$, and \label{item:h1h2ii}
        \item for all $(C,D)\in\mct$, $(G[A\cup C],\g)$ does not contain a $\Lambda$-allowable $A$-path, \label{item:h1h2iii}
    \end{enumerate}
    then there is a compact $w$-subwall $W''$ of $W'$, $A$-$W''$-linkages $\mcq_1,\mcq_2$ with $|\mcq_1|=|\mcq_2|=t$ that nicely link to $W''$, and cosets $H_1,H_2$ of $\G'$ such that every path in $\mcq_i$ has $\g$-length in $H_i$ for $i\in[2]$ and $(H_1+H_2)\cap \Lambda\neq\emptyset$. 
\end{lemma}
\begin{proof}

Define
\begin{align*}
w_0 &= w, \\
w_i &= (2w_{i-1}+1)(2T+1)|\G|+2t \text{ for } i\geq 1, \text{ and} \\
\beta_{\ref{lem:H1H2linkages}}(w,t,|\G|) &= w_{|\G|}+2.
\end{align*}
Note that $\beta_{\ref{lem:H1H2linkages}}(w,t,|\G|) > 12T|\G|$.

Let $(G,\g)$ be a $\G$-labelled graph with $A\subseteq V(G)$ satisfying \ref{item:h1h2i}, and let $W'$ be a $\beta_{\ref{lem:H1H2linkages}}(w,t,|\G|)$-wall in $(G-A,\g)$ satisfying \ref{item:h1h2ii} and \ref{item:h1h2iii}. 
Note that, since $(B,\g_B)$ is a $\G'$-labelled graph by \ref{item:h1h2ii}, we have that every $V(B)$-path in $(G-A,\g)$ has $\g$-length in $\G'$.
Let $W$ be a 1-contained $w_{|\G|}$-subwall of $W'$.

    \begin{claim}\label{claim:H1+H2}
        Let $W^*$ be a compact $w^*$-subwall of $W$ such that $w^*\geq 12T|\G|$. Let $\mch^*\subseteq \G/\G'$ denote the set of cosets $H$ of $\G'$ such that there does not exist $X\subseteq V(G)$ with $|X|<12T$ intersecting every $A$-$N^{W^*}$-path whose $\g$-length is in $H$.
        Then there exist $H_1,H_2\in\mch^*$ such that $(H_1+H_2)\cap \Lambda \neq\emptyset$.
    \end{claim}
    \begin{subproof}
        For each $H\in(\G/\G')\setminus \mch^*$, let $X_H\subseteq V(G)$ be a vertex set with $|X_H|<12T$ intersecting every $A$-$N^W$-path whose $\g$-length is in $H$, and let $X=\bigcup_{H\in(\G/\G')\setminus \mch^*}X_H$.
        Then in $(G-X,\g)$, every $A$-$N^W$-path has $\g$-length in $\cup\mch^*$.
        Since $|X|<12T|\G|$, $X$ does not intersect all $\Lambda$-allowable $A$-paths by \ref{item:h1h2i}, so there is an $A$-path $P$ in $(G-X,\g)$ with $\g(P)\in \Lambda$.
        Note that $P$ intersects $V(B)$, since otherwise $P$ would be contained in the small side of a separation of order at most $|X|+2$, violating \ref{item:h1h2iii}.
        Let $S_1,S_2$ denote the two $A$-$V(B)$-subpaths of $P$. 
        Then the edges in $E(P)-E(S_1\cup S_2)$ form a $V(B)$-path in $(G-A-X,\g)$, and its $\g$-length is in $\G'$  by \ref{item:h1h2ii}.
        Hence, we have $\g(P)=\g(S_1)+\g(S_2)+\g(P-E(S_1\cup S_2)) \in \g(S_1)+\g(S_2)+\G'$.

        Since $|X|<12T|\G|$ and $W^*$ has order $w^*\geq 12T|\G|$, for each $j\in[2]$, there exists a path in $(G-A-X,\g)$ that extends $S_j$ to $N^{W^*}$, and this path has $\g$-length in the coset $H_j:=\g(S_j)+\G'$ because every $V(B)$-path in $(G-A,\g)$ has $\g$-length in $\G'$ by \ref{item:h1h2ii}. 
        Hence $H_1,H_2\in \mch^*$ and $\g(P)\in (\g(S_1)+\g(S_2)+\G')\cap \Lambda = (H_1+H_2)\cap\Lambda$, as desired.
    \end{subproof}

    \begin{claim} \label{claim:AWpathslengthH}
        Let $W^*$ be a compact $w^*$-subwall of $W$ such that $w^*\geq w_i$ for some integer $i\geq 1$.
        Let $\mch^*\subseteq \G/\G'$ denote the set of cosets $H$ of $\G'$ such that there does not exist $X\subseteq V(G)$ with $|X|<12T$ intersecting every $A$-$N^{W^*}$-path whose $\g$-length is in $H$.
        Then for all $H\in\mch^*$, there exists a compact $w_{i-1}$-subwall $W^*_1$ of $W^*$ such that there are $t$ disjoint $A$-$W^*_1$-paths whose $\g$-lengths are in $H$ that nicely link to $W^*_1$.
    \end{claim}
    \begin{subproof}
        Since $w^*\geq w_i \geq (2w_{i-1}+1)(2T+1)+2t$, there is a $t$-contained compact $(2w_{i-1}+1)(2T+1)$-subwall $W^*_0$ of $W^*$. 
        Let $H\in \mch^*$. 
        Note that the $\mct$-large 3-block $(B,\g_B)$ of $(G-A,\g)$ is a $\G'$-labelled graph by \ref{item:h1h2ii}, and $W^*_0$ is 1-contained in $W'$ (because it is $t$-contained in $W^*$ which is a subwall of $W'$).
        Hence, we have $N^{W^*_0}\subseteq \branch(W')\subseteq V(B)$ by the definitions of $\mct=\mct_{W'}$ and $\mct$-large 3-blocks. Moreover, there does not exist $X\subseteq V(G)$ with $|X|<12T$ intersecting all $A$-$N^{W^*_0}$-paths whose $\g$-length is in $H$ by assumption.
        Hence, by Lemma \ref{lem:AWpathslengthL} applied to $\G/\G'$, $W^*_0$ contains a compact $w_{i-1}$-subwall $W^*_1$ such that there are $t$ disjoint $A$-$W^*_1$-paths whose $\g$-lengths are in $H$ that nicely link to $W^*_1$.
    \end{subproof}
    Let $W_0=W'$ and $\mch_0=\emptyset$. For $i=0,1,2,\dots,|\G/\G'|-1$, in this order, apply the following procedure:
    \begin{enumerate}
        \item Let $\mch_i'\subseteq \G/\G'$ denote the set of cosets $H$ of $\G'$ such that there does not exist $X\subseteq V(G)$ with $|X|<12T$ intersecting every $A$-$N^{W_i}$-path whose $\g$-length is in $H$. 
        \item 
        Define $\mch_{i+1}=\mch_i\cup\{H_{i+1}\}$, where $H_{i+1}$ is defined as follows. Note that by Claim \ref{claim:H1+H2}, there exist $H,H'\in \mch_i'$ such that $(H+H')\cap \Lambda\neq\emptyset$.
        \begin{enumerate}
            \item If possible, choose such a pair $H,H'$ so that $H\in\mch_i$, and define $H_{i+1}=H'$.
            \item Otherwise, choose such a pair $H,H'\in\mch_i'\setminus\mch_i$ arbitrarily, and define $H_{i+1}=H$.
        \end{enumerate}
        \item Apply Claim \ref{claim:AWpathslengthH} to $W_i$ and $H_{i+1}$ to obtain a compact $w_{|\G|-i-1}$-subwall $W_{i+1}$ of $W_i$ such that there is a linkage $\mcq_{i+1}$ of $t$ disjoint $A$-$W_{i+1}$-paths whose $\g$-lengths are in $H_{i+1}$ that nicely link to $W_{i+1}$.
    \end{enumerate}
    Let $W''=W_{|\G/\G'|}$ and $\mch''=\mch_{|\G/\G'|}=\{H_1,\dots,H_{|\G/\G'|}\}$. Note that $W''$ is a compact subwall of $W'$ of order $w_{|\G|-|\G/\G'|} \geq  w_0 = w$, and that $W''$ is strongly balanced because it is a subwall of $W'$.
    \begin{claim}
        There exist $i,j\in[|\G/\G'|]$ such that $(H_i+H_j)\cap\Lambda\neq\emptyset$.
    \end{claim}
    \begin{subproof}
        If (2)(a) was applied in any iteration $i=0,1,\dots,|\G/\G'|-1$, then the pair $H,H'$ in that step both belong to $\mch_{i+1}$, and the claim immediately follows. So we may assume that (2)(b) was applied in every iteration.
        This implies that $|\mch''|=|\G/\G'|$; that is, $\mch''$ contains every coset of $\G'$, and in particular $\mch_{|\G/\G'|-1} = (\G/\G')\setminus\{H_{|\G/\G'|}\}$.
        Moreover, $H_{|\G/\G'|}$ was defined to be $H$, where we had $H,H'\in\mch_{|\G/\G'|-1}'\setminus \mch_{|\G/\G'|-1}$ such that $(H+H')\cap\Lambda\neq\emptyset$. This implies that $H=H'$ and the claim follows with $i=j=|\G/\G'|$.
    \end{subproof}
    Let $i,j$ be as in the claim, and assume without loss of generality that $i\leq j$.
    We have that $\mcq_i$ is an $A$-$W_i$-linkage that nicely links to $W_i$ such that every path in $\mcq_i$ has $\g$-length in $H_i$, $\mcq_j$ is an $A$-$W_j$-linkage that nicely links to $W_j$ such that every path in $\mcq_j$ has $\g$-length in $H_j$, and $(H_i+H_j)\cap\Lambda\neq\emptyset$.

    If $W_i=W''$, define $\mcq_i'=\mcq_i$.
    If $W_i\neq W''$, then $W''$ is $t$-contained in $W_i$, so by Menger's theorem we can extend $\mcq_i$ through $W_i$ to obtain an $A$-$W''$-linkage $\mcq_i'$ that nicely links to $W''$. Since $W''$ is strongly balanced, every path in $\mcq_i'$ still has $\g$-length in $H_i$.
    Define $\mcq_j'$ similarly.

    We now have that $\mcq_i',\mcq_j'$ are $A$-$W''$-linkages with $|\mcq_i'|=|\mcq_j'|=t$ that nicely link to $W''$ such that every path in $\mcq_i'$ (resp. $\mcq_j'$) has $\g$-length in $H_i$ (resp. $H_j$), and $(H_i+H_j)\cap\Lambda\neq\emptyset$, as desired.
\end{proof}

\begin{lemma}\label{lem:PQdisjointwithsame12attachments}
    Let $G$ be graph with $A\subseteq V(G)$, let $W$ be a wall in $G-A$, let $B$ denote the $\mct_W$-large 3-block of $G-A$, and let $\mcp$ be a linkage in $G-A$ such that every path in $\mcp$ has at least one endpoint in $\partial N^W$ and is internally disjoint from $W$.
    Let $\mcq$ be an $A$-$W$-linkage that nicely links to $W$.
    Then there exist disjoint linkages $\mcp',\mcq'$ such that 
    \begin{enumerate}[label=(\roman*)]
        \item  $\mcp'\subseteq \mcp$ and $|\mcp'|\geq |\mcp|-2|\mcq|$, \label{item:pqi}
        \item $\mcq'$ is an $A$-$W$-linkage with $|\mcq'|=|\mcq|$, and \label{item:pqii}
        \item  for each $Q'\in\mcq'$, there exists $Q\in\mcq$ with the same $A$-$V(B)$-subpath as $Q'$, and the endpoint of $Q'$ in $W$ is an endpoint of a path in $\mcp\cup\mcq$.\label{item:pqiii}
    \end{enumerate} 
\end{lemma}
\begin{proof}
    Note that $\mcq$ satisfies \ref{item:pqii} and \ref{item:pqiii}. 
    Let $\mcq'$ be a linkage satisfying \ref{item:pqii} and \ref{item:pqiii} such that the number of edges in $\mcq'$ not in $\cup\mcp$ is minimized.
    \begin{claim}
        There are at most $2|\mcq'|$ paths in $\mcp$ that intersect $\cup\mcq'$.
    \end{claim}
    \begin{subproof}
        Let $P_1,\dots,P_m\in\mcp$ denote the paths in $\mcp$ that intersect $\cup\mcq'$. For each $i\in[m]$, let $w_i$ be an endpoint of $P_i$ in $\partial N^W$ and let $q_i$ be the vertex in $P_i\cap \cup\mcq'$ closest to $w_i$ in $P_i$.
        If $m>2|\mcq'|$, then there exists $Q\in\mcq'$ and three distinct indices $i,j,k\in[m]$ such that $q_i,q_j,q_k\in V(Q)$.
        Let $v$ denote the endpoint of $Q$ in $A$, and let us assume without loss of generality that the vertices $v, q_i,q_j,q_k$ occur in this order on the path $Q$.        
        Observe that $q_j\in V(B)$ because there are three paths (namely $q_jQq_iP_iw_i$, $q_jP_jw_j$, and $q_jQq_kP_kw_k$), mutually disjoint except at $q_j$, from $q_j$ to $N^W$ (and hence to $V(B)$).
        Now let $Q''=vQq_jP_jw_j$ and $\mcq'':=(\mcq'\setminus\{Q\})\cup\{Q''\}$.
        Then $Q''$ is an $A$-$W$-path disjoint from every path in $\mcq'\setminus\{Q\}$ with the same $A$-$V(B)$-subpath as $Q$, and the endpoint of $Q''$ in $W$ is an endpoint of the path $P_j\in \mcp\cup\mcq$, so $\mcq''$ satisfies \ref{item:pqii} and \ref{item:pqiii}. On the other hand, $Q''$ has fewer edges not in $\cup\mcp$ than $Q$,  contradicting  our choice of $\mcq'$.
    \end{subproof}
    Hence there exists $\mcp'\subseteq\mcp$ with $|\mcp'|\geq|\mcp|-2|\mcq|$ such that every path in $\mcp'$ is disjoint from every path in $\mcq'$. Then $\mcp'$ and $\mcq'$ satisfy \ref{item:pqi}, \ref{item:pqii}, and \ref{item:pqiii}.
\end{proof}

\begin{lemma}
    \label{lem:generatingwallribbon}
    There exists a function $f_{\ref{lem:generatingwallribbon}}:\mbn^2\to\mbn$ satisfying the following.
    Let $\G$ be a finite abelian group and let $k$ be a positive integer.
    Let $(G,\g)$ be a $\G$-labelled graph with $A\subseteq V(G)$.
    Let $W$ be a wall in $G-A$ and let $\mcp_1,\dots,\mcp_m$ be $W$-handlebars in $G-A$ such that $(W,\{\mcp_1,\dots,\mcp_m\})$ is $\G'$-generating for some subgroup $\G'\le\G$ and the $\mct_W$-large 3-block $(B,\g_B)$ of $(G-A,\g)$ is a $\G'$-labelled graph.
    Let $\mcq_1,\mcq_2$ be $A$-$W$-linkages such that the endpoint of every path in $\mcq_1\cup\mcq_2$ in $W$ is on the column boundary of $W$, and suppose there exist cosets $H_1,H_2$ of $\G'$ such that for $j\in[2]$, every path in $\mcq_j$ has $\g$-length in $H_j$.
    If the linkages $\mcp_1,\dots,\mcp_m,\mcq_1,\mcq_2$ each have size at least $f_{\ref{lem:generatingwallribbon}}(k,|\G|)$, then there exist $W$-handlebars $\mcp_1^*,\dots,\mcp_m^*$ and $A$-$W$-handlebars $\mcq_1^*,\mcq_2^*$ such that $(W,\{\mcp_1^*,\dots,\mcp_m^*\})$ is $\G'$-generating and $(W,(\mcp_1^*,\dots,\mcp_m^*),\mcq_1^*,\mcq_2^*,\g)$ is an $(A,k)$-ribboned wall. 
\end{lemma}
\begin{proof}
    Define $k'=(f_{\ref{lem:handlebarsnonmixing}}(|\G|,2k)-1)^3+1$ and $f_{\ref{lem:generatingwallribbon}}(k,|\G|) = 17k'|\G|$. Suppose that $\mcp_1,\dots,\mcp_m,\mcq_1,\mcq_2$ each have size at least $f_{\ref{lem:generatingwallribbon}}(k,|\G|)$.
    We delete some paths from each linkage so that $|\mcp_i|=17k'|\G|$ for all $i\in[m]$, $|\mcq_1|=6k'|\G|$, and $|\mcq_2|=2k'|\G|$.
    
    Let $\mcp = \mcp_1\cup\dots\cup\mcp_m$. 
    Applying Lemma \ref{lem:PQdisjointwithsame12attachments} to $W$, $\mcp$, and $\mcq_1$, we obtain disjoint linkages $\mcp',\mcq_1'$ such that
    \begin{enumerate}[label=(1.\roman*)]
        \item  $\mcp'\subseteq \mcp$ and $|\mcp'|\geq |\mcp|-2|\mcq_1|=|\mcp|-12k'|\G|$, \label{item:1pqi}
        \item $\mcq_1'$ is an $A$-$W$-linkage with $|\mcq_1'|=|\mcq_1|$, and \label{item:1pqii}
        \item  for each $Q'\in\mcq_1'$, there exists $Q\in\mcq_1$ with the same $A$-$V(B)$-subpath as $Q'$, and the endpoint of $Q'$ in $W$ is an endpoint of a path in $\mcp\cup\mcq_1$.\label{item:1pqiii}
    \end{enumerate} 
    Since $(B,\g_B)$ is $\G'$-balanced, \ref{item:1pqiii} implies that the $\g$-length of every path in $\mcq_1'$ is in the coset $H_1$. Since $|\mcp_i|=17k'|\G|$ for all $i\in[m]$, by \ref{item:1pqi}, there exist $W$-handlebars $\mcp_1',\dots,\mcp_m'$ such that for all $i\in[m]$, $|\mcp_i'|=5k'|\G|$ and $\mcp_i'\subseteq\mcp_i$ (and hence $\mcp_i'$ is disjoint from $\mcq_1'$).

    We again apply Lemma \ref{lem:PQdisjointwithsame12attachments} to $W$, $\mcp'':=(\mcp_1'\cup\dots\cup\mcp_m'\cup\mcq_1')$, and $\mcq_2$ to obtain disjoint linkages $\mcp''', \mcq_2''$ such that
    \begin{enumerate}[label=(2.\roman*)]
        \item  $\mcp'''\subseteq \mcp''$ and $|\mcp'''|\geq |\mcp''|-2|\mcq_2|=|\mcp''|-4k'|\G|$, \label{item:2pqi}
        \item $\mcq_2''$ is an $A$-$W$-linkage with $|\mcq_2''|=|\mcq_2|$, and \label{item:2pqii}
        \item  for each $Q'\in\mcq_2''$, there exists $Q\in\mcq_2$ with the same $A$-$V(B)$-subpath as $Q'$, and the endpoint of $Q'$ in $W$ is an endpoint of a path in $\mcp''\cup\mcq_2''$.\label{item:2pqiii}
    \end{enumerate}     
    By \ref{item:2pqiii}, the $\g$-length of every path in $\mcq_2''$ is in the coset $H_2$. Since $|\mcq_1'|=6k'|\G|$ and $|\mcp_i'|=5k'|\G|$ for all $i\in[m]$, by \ref{item:2pqi}, there exist $W$-handlebars $\mcp_1'',\dots,\mcp_m'', \mcq_1''$ such that $|\mcq_1''|=2k'|\G|$, $\mcq_1''\subseteq\mcq_1'$, and for all $i\in[m]$, $|\mcp_i''|=k'|\G|$ and $\mcp_i''\subseteq\mcp_i'$.

    Since $|\mcq_1''|=|\mcq_2''|=2k'|\G|$ and since the endpoint of every path in $\mcq_1\cup\mcq_2$ in $W$ is on the column boundary of $W$ by \ref{item:1pqiii} and \ref{item:2pqiii}, there exist subsets $\mcq_1'''\subseteq\mcq_1''$ and $\mcq_2'''\subseteq\mcq_2''$ such that, for $j\in[2]$, we have  $|\mcq_j'''|=k'|\G|$ and the endpoints of the paths in $\mcq_j'''$ in $W$ are either all on the first column of $W$ or all on the last column of $W$.
    Moreover, since every path in $\mcq_j$ has $\g$-length in $H_j$ and $|\mcq_j'''| = k'|\G|$, there exists $\mcq_j^1\subseteq \mcq_j'''$ with $|\mcq_j^1|=k'$ such that the paths in $\mcq_j^1$ have the same $\g$-length, say $h_j\in H_j$.
    Now let $\mcp_{m+1}$ be a set of $W$-handles with $|\mcp_{m+1}|=k'$ obtained by joining $\mcq_1^1$ and $\mcq_2^1$ arbitrarily.
    Since $k' = (f_{\ref{lem:handlebarsnonmixing}}(|\G|,2k)-1)^3+1$,
    by Lemma \ref{lem:purelinkages}, there is a subset $\mcp_{m+1}^1$ of $\mcp_{m+1}$ with $|\mcp_{m+1}^1|=f_{\ref{lem:handlebarsnonmixing}}(|\G|,2k)$ such that $\mcp_{m+1}^1$ is a $W$-handlebar.
    
    For $i\in[m]$, let $\mcp_i^1$ be a subset of $\mcp_i''$ with $|\mcp_i^1|=f_{\ref{lem:handlebarsnonmixing}}(|\G|,2k)$.
    Then $\mcp_1^1,\dots,\mcp_{m+1}^1$ are pairwise disjoint sets of $f_{\ref{lem:handlebarsnonmixing}}(|\G|,2k)$ disjoint $W$-handles.
    Since $m<|\G|$, we have $f_{\ref{lem:handlebarsnonmixing}}(|\G|,2k)\geq f_{\ref{lem:handlebarsnonmixing}}(m+1,2k)$, so by Lemma \ref{lem:handlebarsnonmixing}, there exist $W$-handlebars $\mcp_1^*,\dots,\mcp_{m+1}^*$ that are pairwise non-mixing such that for all $i\in[m+1]$, $\mcp_i^*\subseteq\mcp_i^1$ and $|\mcp_i^*|=2k$.
    By splitting $\mcp_{m+1}^*$, we obtain two $A$-$W$-linkages $\mcq_1^*,\mcq_2^*$ such that for $j\in[2]$, $\mcq_j^*\subseteq \mcq_j^1$ and $|\mcq_j^*|=2k$. 
    
    Since $\mcp_1^*,\dots,\mcp_{m+1}^*$ are pairwise disjoint and non-mixing with respect to $W$, so are $\mcp_1^*,\dots,\mcp_m^*$, $\mcq_1^*$, $\mcq_2^*$, so \ref{item:Aobstructions-union} is satisfied. Since $(W,\{\mcp_1,\dots,\mcp_m\})$ is $\G'$-generating and $\mcp_i^*\subseteq \mcp_i$ for $i\in[m]$, and since every path in $\mcq_j^*$ has $\g$-length $h_j$ for $j\in[2]$, \ref{item:Aobstructions-zerowall} is satisfied as well.
    Therefore, $(W,(\mcp_i^*:i\in[m]),\mcq_1^*,\mcq_2^*,\g)$ is an $(A,k)$-ribboned wall.    
\end{proof}

\begin{lemma}\label{lem:makeirreducible}
    There exists a function $f_{\ref{lem:makeirreducible}}:\mbn^2\to\mbn$ satisfying the following.
    Let $\G$ be a finite abelian group, let $\Lambda\subseteq\G$, and let $k$ be a positive integer.     
    Let $\mcr=(W,(\mcp_i:i\in[m]),\mcq_1,\mcq_2,\g)$ be an $(A,f_{\ref{lem:makeirreducible}}(k,|\G|))$-ribboned $\theta$-wall such that $(W,\{\mcp_1,\dots,\mcp_m\})$ is $\G'$-generating for some subgroup $\G'\le\G$ and $(h_1+h_2+\G') \cap\Lambda\neq\emptyset$, where $h_j$ is the $\g$-length of every path in $\mcq_j$.
    Then $G(\mcr)$ contains a $\Lambda$-irreducible $(A,k)$-ribboned $(\theta-2|\G|)$-wall. 
\end{lemma}

\begin{proof}
    By repeatedly removing indices $i'\in[m]$ for which there exist $b_1,b_2\in\{0,1,2\}$ with $b_1+b_2=2$ such that $((b_1h_1+b_2h_2)+\langle g_i: i\in[m]-\{i'\}\rangle) \cap \Lambda\neq\emptyset$, we may assume without loss of generality that
    \begin{equation}\tag{$*$}\label{eqn:A4}
        \parbox{\dimexpr\linewidth-4em}{for each $i'\in[m]$ and $b_1,b_2\in\{0,1,2\}$ with $b_1+b_2=2$, we have $((b_1h_1+b_2h_2)+\langle g_i: i\in[m]-\{i'\}\rangle) \cap \Lambda=\emptyset$.}
    \end{equation}
    Moreover, if for some $j\in[2]$, there exist integers $d_1,\dots,d_m$ such that $2h_j+\sum_{i\in[m]}d_ig_i\in\Lambda$, then we set both $\mcq_1$ and $\mcq_2$ to be equal to $\mcq_j$.

    Define $f_{\ref{lem:makeirreducible}}(k,|\G|) = 2k|\G|^{|\G|}$.
    Then we have $|\mcq_j|\geq 4k|\G|^{|\G|}$ for each $j\in[2]$.
    Let $\mcp_{m+1}$ be a $W$-handlebar with $|\mcp_{m+1}|\geq 2k|\G|^{|\G|}$ obtained by joining $\mcq_1$ and $\mcq_2$. 
    Since $(W,\{\mcp_1,\dots,\mcp_m\})$ is $\G'$-generating and since $(h_1+h_2+\G')\cap\Lambda\neq\emptyset$, there exist non-negative integers $d_1,\dots,d_m$ with $d_i<|\G|$ such that $h_1+h_2+\sum_{i\in[m]}d_ig_i\in\Lambda$.
    By \eqref{eqn:A4}, the integers $d_1,\dots,d_m$ are all positive.
    
    Let $W^0=W$ and $\mcp_i^0=\mcp_i$ for $i\in[m+1]$. For $i\in[m]$, define $g_i^0=g_i$.
    For $j=1,2,\dots,m$, in this order, apply the following procedure.
    \begin{enumerate}[label=(\arabic*)]
        \item Let $W^j$ be a $(\theta-2j)$-column-slice of $W^{j-1}$ disjoint from the first and last columns of $W^{j-1}$. \label{item:comb1}
        \item Note that there exist positive integers $d_1,\dots,d_m$ with $d_i<|\G|$ such that $h_1+h_2+\sum_{i\in[m]}d_ig_i^{j-1}\in\Lambda$.
        If possible, choose such integers $d_1,\dots,d_m$ so that $d_j$ is even, where $j$ is the current index. Let $d_{m+1}=1$. \label{item:comb2}
        \item Note that $|\mcp_i^{j-1}|\geq 2k|\G|^{|\G|-(j-1)}$ for all $i\in[m+1]$. Apply Lemma \ref{lem:combining-handles} to obtain pairwise disjoint and non-mixing $W^j$-handlebars $\mcp_1^j,\dots,\mcp_{m+1}^j$ in $H$ such that for each $i\in[m+1]$, every path in $\mcp_i^j$ contains exactly $d_i$ paths in $\mcp_i^{j-1}$ and $|\mcp_i^j| =2k|\G|^{|\G|-j}$, and such that if $d_i$ is even, then $\mcp_i^j$ is in series, and if $d_i$ is odd, then $\mcp_i^j$ is of the same type as $\mcp_i^{j-1}$. 
        Define $g_i^j = d_ig_i^{j-1}$, so that every path $\mcp_i^j$ has $\g$-length $g_i^j$. 
        \label{item:comb3}
    \end{enumerate}
    Since $(W,\{\mcp_1,\dots,\mcp_m\})$ is $\G'$-generating, it follows from \ref{item:comb2} and \ref{item:comb3} that 
    \begin{itemize}
        \item for each $j\in[m]$, we have $h_1+h_2+\sum_{i\in[m]}g_i^j\in\Lambda$ and,
        \item for each $i\in [j]$, either $\mcp_i^j$ is in series or there does not exist positive integers $d_1',\dots,d_m'$ with $d_i'$ even such that $h_1+h_2+\sum_{i\in[m]}d_i'g_i^j\in\Lambda$.
    \end{itemize} 

    Define $W^*=W^m$, $\mcp_i^* = \mcp_i^m$ for each $i\in[m]$, and let $\mcq_1^*,\mcq_2^*$ denote the two $A$-$W$-handlebars obtained by splitting $\mcp_{m+1}^m$. 
    Then $|\mcq_j^*|\geq 2k|\G|^{|\G|-m}\geq 2k$ for each $j\in[2]$, and $|\mcp_i^*|\geq 2k|\G|^{|\G|-m}\geq 2k$ for each $i\in[m]$. Moreover, $W^*$ is a $(\theta-2m)$-column-slice of $W$ by \ref{item:comb1}.  
    
    Since $\mcq_1$ and $\mcq_2$ are either disjoint or equal, so are $\mcq_1^*$ and $\mcq_2^*$.
    Moreover, since $\mcp_1^m,\dots,\mcp_{m+1}^m$ are pairwise disjoint and non-mixing, so are $\mcp_1^*,\dots,\mcp_m^*,\mcq_1^*\cup\mcq_2^*$.
    Hence \ref{item:Aobstructions-union} is satisfied.
    Since $(W,\{\mcp_1,\dots,\mcp_m\})$ is $\G'$-generating, it follows from \ref{item:comb3} that \ref{item:Aobstructions-zerowall} is satisfied. Hence, $\mcr^*:=(W^*,(\mcp_i^*:i\in[m]),\mcq_1^*,\mcq_2^*,\g)$ is an $(A,k)$-ribboned $(\theta-2m)$-wall.
    Define $g_i^* = g_i^m$ and $h_j^*=h_j$ for $i\in[m]$ and $j\in[2]$.
    Then every path in $\mcp_i^*$ and $\mcq_j^*$ has length $g_i^*$ and $h_j^*$, respectively, for $i\in[m]$ and $j\in[2]$.

    By \ref{item:comb2} and \ref{item:comb3}, we have $(\sum_{i\in[m]}g_i^*)+h_1^*+h_2^*\in \Lambda$, so \ref{item:Aobstructions-allowabletransversals} is satisfied.
    By \eqref{eqn:A4}, \ref{item:Aobstructions-minimality} is  satisfied.
    It also follows from \ref{item:comb2} and \ref{item:comb3} that \ref{item:Aobstructions-even} is satisfied. Therefore, $\mcr^*$ is $\Lambda$-irreducible.
\end{proof}

\section{Proof of Theorem \ref{thm:mainApathstructure}} \label{sec:proof}
We now complete the proof of Theorem \ref{thm:mainApathstructure}, restated here for convenience.
\mainApathstructure*
\begin{proof}
    
    For each positive integer $k$, define
    \begin{align*}
    t=t(k,|\G|)&=f_{\ref{lem:generatingwallribbon}}(f_{\ref{lem:makeirreducible}}(k,|\G|),|\G|), \\
    T=T(k,|\G|)&=3(36t)^2, \\
    w=w(k,|\G|)&= 12t|\G|, \\
    \theta = \theta(k,|\G|) &= f_{\ref{lem:generatinglinkagessubwall}}(\beta_{\ref{lem:H1H2linkages}}(w,t,|\G|),t,|\G|),
    \end{align*}
    and let $f:\mbn^2\to\mbn$ be a function such that $f(0,|\G|)=0$ and
    \begin{align*}
    f(k,|\G|) \geq 2f(k-1,|\G|)+3f_{\ref{thm:tanglewall}}(\theta)+10 + g_{\ref{lem:generatinglinkagessubwall}}(t,|\G|)+ 12T|\G|
    \end{align*}
    for all $k\geq 1$.

    Suppose that $((G,\g),k)$ is a minimal counterexample to $f_{\G}(k):=f(k,|\G|)$ being an Erd\H{o}s-P\'osa function for the family of $\Lambda$-allowable $A$-paths.

    By Lemma \ref{lem:minctextangle}, $(G-A,\g)$ admits a tangle $\mct$ of order $f_{\ref{thm:tanglewall}}(\theta)$ such that for each $(C,D)\in\mct$, $(G[A\cup C],\g)$ does not contain a $\Lambda$-allowable $A$-path.

    By Lemma \ref{thm:tanglewall}, there is a wall $W$ in $G-A$ of order at least $\theta=f_{\ref{lem:generatinglinkagessubwall}}(\beta_{\ref{lem:H1H2linkages}}(w,t,|\G|),t,|\G|)$ such that $\mct$ dominates $W$.

    By Lemma \ref{lem:generatinglinkagessubwall}, there is a subgroup $\G'\le \G$ and a vertex set $Z\subseteq V(G)-A$ with $|Z|\leq g_{\ref{lem:generatinglinkagessubwall}}(t,|\G|)$ such that $(G-A-Z,\g)$ contains a strongly balanced $\beta_{\ref{lem:H1H2linkages}}(w,t,|\G|)$-subwall $W'$ of $W$ and a family of pairwise disjoint and non-mixing $W'$-handlebars $\mcp_1,\dots,\mcp_m$ each of size $t$ such that $(W',\{\mcp_1,\dots,\mcp_m\})$ is $\G'$-generating and the $\mct_{W'}$-large 3-block $(B,\g_B)$ of $(G-A-Z,\g)$ is $\G'$-balanced. Note that $m<|\G|$.

    Since $(B,\g_B)$ is $\G'$-balanced, by Lemma \ref{lem:shiftequivalent}, there exists a sequence of shifts in $(B,\g_B)$ resulting in a $\G'$-labelled graph $(B,\g_B')$. Suppose that the same sequence of shifts in $(G,\g)$ results in $(G,\g')$. Then $(B,\g_B')$ is a 3-block of $(G-A-Z,\g')$, so every $V(B)$-path in $G-A-Z$ has $\g'$-length in $\G'$.
    Since shifting at vertices in $V(B)\subseteq V(G)-A$ does not change the length of any $A$-path, we may assume without loss of generality that we have already performed these shifts; that is, we may assume that $(B,\g_B)$ is a $\G'$-labelled graph and every $V(B)$-path in $G-A-Z$ has $\g$-length in $\G'$.

    In particular, this implies that $(G-Z,\g)$ and $W'$ satisfy the hypothesis \ref{item:h1h2ii} of Lemma \ref{lem:H1H2linkages}. Also note that \ref{item:h1h2iii} is satisfied since $\mct$ dominates $W'$ and $\mct$ satisfies \ref{item:h1h2iii} by Lemma \ref{lem:minctextangle}. Furthermore, there does not exist $Y\subseteq V(G-Z)$ with $|Y|<12T|\G|$ intersecting every $\Lambda$-allowable $A$-path in $(G-Z,\g)$ because this would imply that $Z\cup Y$ is a vertex set in $(G,\g)$ with $|Z\cup Y|<g_{\ref{lem:generatinglinkagessubwall}}(t,|\G|)+12T|\G|<f(k,|\G|)$ intersecting every $\Lambda$-allowable $A$-path in $(G,\g)$, contradicting the assumption that $((G,\g),k)$ is a minimal counterexample to $f(\cdot,|\G|)$ being an Erd\H{o}s-P\'osa function for the family of $\Lambda$-allowable $A$-paths.
    Hence, \ref{item:h1h2i} is satisfied.

    By Lemma \ref{lem:H1H2linkages}, there is a compact $w$-subwall $W''$ of $W'$, $A$-$W''$-linkages $\mcq_1,\mcq_2$ with $|\mcq_1|=|\mcq_2|=t$ that nicely link to $W''$, and cosets $H_1,H_2$ of $\G'$ such that every path in $\mcq_i$ has $\g$-length in $H_i$ for $i\in[2]$ and $(H_1+H_2)\cap\Lambda\neq\emptyset$.

    Let $W'''$ be a $2t|\G|$-contained $2t|\G|$-subwall of $W''$. Then $W'''$ is $2t|\G|$-contained in $W'$. Let $U'$ denote the union of the set of endpoints of the paths in $\mcp_1\cup\dots\cup\mcp_m$ and the set of endpoints of the paths in $\mcq_1\cup\mcq_2$ in $W''$, and let $V'$ denote the column boundary of $W'''$. Note that $|U'|\leq 2tm+2t\leq 2t|\G|$ and $|V'| \geq 4t|\G|$. Since $W'''$ is $2t|\G|$-contained in $W'$, there does not exist a set of fewer than $|U'|$ vertices separating $U'$ and $V'$ in $W'-(V(W''')-V')$. By Menger's theorem, there exists a set $\mcs$ of $|U'|$ disjoint paths in $W'-(V(W''')-V')$ from $U'$ to $V'$.
    It is easy to see that $\mcs$ can be chosen so that, for every pair $u'\in U'$ and $v'\in V'$ of endpoints of a path in $\mcs$, if $u'$ is in the first or last column of $W'$, then $v'$ is in the first or last column of $W'''$ respectively.
    Let us assume without loss of generality that $\mcs$ is chosen this way.
    
    Let $\mcp_1',\dots,\mcp_m'$ be obtained from $\mcp_1,\dots,\mcp_m$ by extending each $W'$-handle along two paths in $\mcs$ so that it becomes $W'''$-handle. By the planarity of $W'$, the cyclic ordering of the endpoints of the paths in $\mcp_1'\cup\dots\cup\mcp_m'$ around $\partial W'''$ is inherited from the cyclic ordering of the endpoints of the paths in $\mcp_1\cup\dots\cup\mcp_m$ around $\partial W'$. Together with our choice of $\mcs$, this implies that $\mcp_1',\dots,\mcp_m'$ are pairwise disjoint non-mixing $W'''$-handlebars.
    Moreover, since $W'$ is strongly balanced, the paths in $\mcp_i'$ have the same $\g$-length as the paths in $\mcp_i$.
    Hence, $(W''',\{\mcp_1',\dots,\mcp_m'\})$ is $\G'$-generating.
    Also note that $\mct_{W'}$ dominates $W'''$, so $(B,\g_B)$ is also the $\mct_{W'''}$-large 3-block of $(G-A-Z,\g)$.
    Then $|\mcp_1'|=\dots=|\mcp_m'|=|\mcq_1|=|\mcq_2|=t$. Recall that $t=f_{\ref{lem:generatingwallribbon}}(f_{\ref{lem:makeirreducible}}(k,|\G|),|\G|)$.

    By Lemma \ref{lem:generatingwallribbon} applied to $W'''$, $\mcp_1',\dots,\mcp_m'$, and $\mcq_1,\mcq_2$, there exist $W'''$-handlebars $\mcp_1^*,\dots,\mcp_m^*$ and $A$-$W$-handlebars $\mcq_1^*,\mcq_2^*$ such that $(W,\{\mcp_1^*,\dots,\mcp_m^*\})$ is $\G'$-generating and \[\mcr^*:=(W,(\mcp_1^*,\dots,\mcp_m^*),\mcq_1^*,\mcq_2^*,\g)\] is an $(A,f_{\ref{lem:makeirreducible}}(k,|\G|)$)-ribboned wall.

    By Lemma \ref{lem:makeirreducible}, $G(\mcr^*)$ (and hence $(G-A,\g)$) contains a $\Lambda$-irreducible $(A,k)$-ribboned wall.
\end{proof}

\bibliographystyle{abbrv}
\bibliography{refs}

\end{document}

%% file: FigObs1.tex
\begin{tikzpicture}

\colorlet{hellgrau}{black!30!white}

\tikzstyle{smallvx}=[thick,circle,inner sep=0.cm, minimum size=1.5mm, fill=white, draw=black]
\tikzstyle{smallblackvx}=[thick,circle,inner sep=0.cm, minimum size=1.5mm, fill=black, draw=black]
\tikzstyle{squarevx}=[thick,rectangle,inner sep=0.cm, minimum size=2mm, fill=white, draw=black]
\tikzstyle{hedge}=[line width=1pt]
\tikzstyle{markline}=[draw=hellgrau,line width=6pt]

\def\gridheight{7}
\def\brickheight{0.7}

\pgfmathtruncatemacro{\lastrow}{\gridheight}
\pgfmathtruncatemacro{\penultimaterow}{\gridheight-1}
\pgfmathtruncatemacro{\lastrowshift}{mod(\gridheight,2)}
\pgfmathtruncatemacro{\lastx}{2*\gridheight+1}

\foreach \i in {0,...,\gridheight}{
  \draw[hedge] (0,\i*\brickheight) -- (\gridheight*\brickheight, \i*\brickheight);
  \draw[hedge] (\i*\brickheight,0) -- (\i*\brickheight, \gridheight*\brickheight);
}
\foreach \i in {0,...,\gridheight}{
  \draw[hedge] (-1.5*\brickheight,\i*\brickheight) edge node[above] {$a$} (0,\i*\brickheight);
  \draw[hedge] (\gridheight*\brickheight,\i*\brickheight) edge node[above] {$b$} (\gridheight*\brickheight+1.5*\brickheight,\i*\brickheight);
}
\foreach \i in {1,...,\gridheight}{
  \draw[hedge] (\i*\brickheight-\brickheight,\gridheight*\brickheight) edge[bend left=80, looseness=2] node[above] {$c$} (\i*\brickheight,\gridheight*\brickheight);
}

\foreach \i in {0,...,\gridheight}{
  \foreach \j in {0,...,\gridheight}{
    \node[smallvx] () at (\i*\brickheight,\j*\brickheight) {};
  }
}
\foreach \i in {0,...,\gridheight}{
  \node[smallblackvx] () at (-1.5*\brickheight,\i*\brickheight) {};
  \node[smallblackvx] () at (\gridheight*\brickheight+1.5*\brickheight, \i*\brickheight) {};
}
\end{tikzpicture}

%% file: FigObs2.tex
\begin{tikzpicture}

\colorlet{hellgrau}{black!30!white}

\tikzstyle{smallvx}=[thick,circle,inner sep=0.cm, minimum size=1.5mm, fill=white, draw=black]
\tikzstyle{smallblackvx}=[thick,circle,inner sep=0.cm, minimum size=1.5mm, fill=black, draw=black]
\tikzstyle{squarevx}=[thick,rectangle,inner sep=0.cm, minimum size=2mm, fill=white, draw=black]
\tikzstyle{hedge}=[line width=1pt]
\tikzstyle{markline}=[draw=hellgrau,line width=6pt]

\def\gridheight{7}
\def\brickheight{0.7}

\pgfmathtruncatemacro{\lastrow}{\gridheight}
\pgfmathtruncatemacro{\penultimaterow}{\gridheight-1}
\pgfmathtruncatemacro{\lastrowshift}{mod(\gridheight,2)}
\pgfmathtruncatemacro{\lastx}{2*\gridheight+1}

\foreach \i in {0,...,\gridheight}{
  \draw[hedge] (0,\i*\brickheight) -- (\gridheight*\brickheight, \i*\brickheight);
  \draw[hedge] (\i*\brickheight,0) -- (\i*\brickheight, \gridheight*\brickheight);
}
\foreach \i in {0,...,\gridheight}{
  \draw[hedge] (-1.5*\brickheight,\i*\brickheight) edge node[above] {$a$} (0,\i*\brickheight);
}
\foreach \i in {1,...,\gridheight}{
  \draw[hedge] (\i*\brickheight-\brickheight,\gridheight*\brickheight) edge[bend left=80, looseness=2] node[above] {$c$} (\i*\brickheight,\gridheight*\brickheight);
  \draw[hedge] (\gridheight*\brickheight, \i*\brickheight-\brickheight) edge[bend right=80, looseness=2] node[right] {$b$} (\gridheight*\brickheight,\i*\brickheight);
}

\foreach \i in {0,...,\gridheight}{
  \foreach \j in {0,...,\gridheight}{
    \node[smallvx] () at (\i*\brickheight,\j*\brickheight) {};
  }
}
\foreach \i in {0,...,\gridheight}{
  \node[smallblackvx] () at (-1.5*\brickheight,\i*\brickheight) {};
}
\end{tikzpicture}